%% file: CuesVelazIarxiv.tex
\documentclass[dvips]{article}
\usepackage{a4wide}
\usepackage[english]{babel}
\usepackage{amsfonts}
\usepackage{amsmath}
\usepackage{amssymb}
\usepackage{graphicx}
\usepackage{color}
\usepackage{pictex}
\providecommand{\U}[1]{\protect\rule{.1in}{.1in}}

\newtheorem{theorem}{Theorem}

\newtheorem{lemma}[theorem]{Lemma}

\newtheorem{remark}[theorem]{Remark}

\newenvironment{proof}[1][Proof]{\noindent\textbf{#1.} }{\ \rule{0.5em}{0.5em}}

\catcode`\@=11

\renewcommand{\theequation}{\thesection.\arabic{equation}}
\@addtoreset{equation}{section}

\begin{document}
\bibliographystyle{abbrv}
\title{
Fluid accumulation in thin-film flows driven by surface 
tension and gravity (I): Rigorous analysis of a drainage equation.
}

\author{
C. M. Cuesta\footnote{Instituto de Ciencias Matem\'{a}ticas CSIC-UAM-UC3M-UCM. 
C/ Nicol\'{a}s Cabrera 15, 28049, Madrid, Spain.},
J.\ J. L. Vel\'{a}zquez\footnote{Institut f\"ur Angewandte Mathematik, Universit\"at Bonn. 
Endenicher Allee 60, 53115 Bonn, Germany.}
}

\date{}
\maketitle
\begin{abstract}
We derive a boundary layer equation describing accumulation regions within a thin-film 
approximation framework where gravity and surface tension balance. As part of the analysis of this 
problem we investigate in detail and rigorously 
the following drainage equation 
\[
\left(\frac{d^3\Phi}{d\tau^3} +1 \right)\Phi^3 = 1 \,.
\]  
In particular, we prove that all solutions that do not satisfy $\Phi\to 1$ as 
$\tau \to\infty$ are oscillatory, and that they oscillate in a very specific way as 
$\tau\to\infty$. This result and the method of proof will be used in 
the analysis of solutions of the afore mentioned boundary layer problem.
\end{abstract}

\section{Introduction.}
The analysis of thin films coating surfaces is a basic problem in fluid 
mechanics and has many applications in industrial processes. In particular flows 
on surfaces with topography had been extensively studied (see, for instance, 
\cite{CrasterMatar}, \cite{Howell}, \cite{JChiniK}, \cite{KalliadasisHomsy}, 
\cite{MyersRev}, \cite{Royetal}, \cite{StockerHosoi}). Usually, the presence 
of regions where the liquid accumulates plays an important role in the overall 
dynamics of these flows. In many situations, the arrival of a thin layer of 
fluid to a region where the liquid accumulates or reduces its velocity, 
results in the formation of oscillations on the fluid interface 
(see e.g. \cite{benilov}, \cite{BertozziPhF}, \cite{bretherton}, \cite{WilsonJones} and references therein). 
On the other hand, it was claimed in \cite{Richards} that the pressure difference between phases 
in unsaturated porous media flow, when complete wetting is 
assumed, is mostly due to the regions where the fluid accumulates. 

Our goal is to derive precise mathematical results concerning the liquid 
accumulation regions mentioned above. To this end, we will study a particular 
type of flows. Namely, we are interested in the motion of thin layers of 
liquid moving on an smooth and, for the moment, arbitrary surface. 
Specifically, our goal is to 
understand the dynamics of such layers of liquid in the case in which the 
forces induced by the curvature of the substrate are comparable to 
gravitational forces. For this particular rescaling of the force terms, the liquid tends to 
accumulate in the regions where the gravitational and surface tension forces 
balance. Outside such regions it is possible to simplify the description of 
the fluid motion using a thin-film approximation. However, this approximation 
breaks down in the regions where the difference of the forces induced by 
gravity and surface tension vanish, these are the regions where the fluid tends to accumulate. 
As a consequence, a boundary layer is required to describe the flow there. 

In this paper, we will restrict our attention to stationary flows in two space dimensions. 
In this case, the thin-film approximation for steady flows reduces, to leading order, to
\begin{equation}
\frac{\partial}{\partial s}\left(Q(s)\, h^{3}\right)=0 \label{steady}
\end{equation} 
where $h$ is the (non-dimensional) height of the fluid above the substrate and $Q(s)$ 
is a function that depends on the geometry of the substrate (parametrised by 
its arc-length $s$) and its relative orientation with respect to the direction 
of gravity. 

As long as the function $Q(s)$ has a constant sign it is 
possible to obtain a stationary solution of (\ref{steady}) for any given flux 
describing the flow. However, if the function $Q(s)$ changes 
sign it is not possible to obtain steady flows satisfying globally the 
assumptions required for the validity of the thin-film approximation. We 
remark, that it is possible to assume in this case, using a convenient 
parametrisation of the curve where the fluid lies, that $Q(s)>0$. 
We shall assume this condition 
in most of the regions where the flow takes place, except in a 
very small region where the liquid will tend to accumulate. In such region 
the difference of curvature between the substrate and the interface, that has 
been neglected in the derivation of (\ref{steady}), becomes relevant. 
Nevertheless, the thin-film approximation remains valid in all of the flow considered here. 

We obtain two main results. First, we derive, using asymptotic approximations, 
an equation describing the height of the liquid in the (small) accumulation regions where 
$Q(s) \leq 0$. This equation reads, in suitable non-dimensional variables,
\begin{equation}
 \left( \frac{d^3 H}{d\xi^3}  + \xi^{2} + a \right)  H^{3} = 1\,, \label{capalimite}
\end{equation}
where $H$ is a rescaled fluid height and $\xi$ is a rescaled arc-length. The 
number $a\in\mathbb{R}$ is a parameter that provides information about the 
size of the region where $Q(s)\leq 0$. Equation (\ref{capalimite}) must be solved subject 
to the corresponding matching conditions, namely, $|\xi|^\frac{2}{3}$ as $|\xi| \to \infty$. 
The existence of such solutions is not \textit{a priori} clear, and their existence will be proved in \cite{CV2}. 
The solutions obtained there are the natural candidates to describe the boundary layer arising at the stagnation 
or accumulation points, i.e. the points $s_0$ where $Q(s_{0})=0$. In the analysis 
of (\ref{capalimite}) it will be essential to have a precise description of the solutions 
of the following auxiliary equation:
\begin{equation}
\left( \frac{d^3\Phi}{d\tau^3} + 1 \right) \Phi^{3}= 1\,. \label{WilsonJones}
\end{equation}
Observe that (\ref{capalimite}) is a non-autonomous third order ODE. 
Equation (\ref{WilsonJones}) results when we introduce new variables 
that transform (\ref{capalimite}) into a fourth order
 autonomous system, and that make the matching conditions also autonomous. 
Namely, in the new independent variable $\tau$, (\ref{capalimite}) reduces to 
(\ref{WilsonJones}) as $|\tau|\to \infty$. That means that the limits $\tau\to -\infty$ 
and $\tau\to \infty$ define two invariant subspaces for (\ref{capalimite}) 
where the flow (of the dynamical system) is given by (\ref{WilsonJones}). In these variables, 
the solutions of (\ref{capalimite}) that satisfy the matching conditions 
correspond to a heteroclinic connection (the critical points connected correspond 
to the matching conditions). The existence proof of these heteroclinic connections is done by 
means of a shooting argument and requires detailed information about the flow 
given by (\ref{WilsonJones}), that is the flow near $|\tau|\to \infty$. 

Our second result in this paper thus concerns the rigorous analysis of (\ref{WilsonJones}).
As we shall see later, one of the main features of (\ref{WilsonJones}) is the 
existence of solutions that oscillate with increasing amplitude as $\tau\to\infty$. 
We thus investigate this oscillatory behaviour rigorously in this paper.

Equation (\ref{WilsonJones}) has an independent physical 
interest. It arises in the description of a thin-film free surface flow on a vertical wall 
that is moving downwards with constant velocity (cf. \cite{WilsonJones}). The existence of oscillatory 
solutions with increasing amplitude was observed in \cite{WilsonJones} using 
formal asymptotic. These oscillations are related to the possibility of air 
entrainment in these flows. 
Due to the fast increase of the amplitude of oscillations, 
the validity of the thin-film approximation breaks down 
and the corresponding solutions need to be matched into other flows. 
In fact, in \cite{WilsonJones} the solutions are matched into  
a stationary pool of fluid. 

We observe that the onset of oscillatory interfaces 
with increasing amplitude seems to be a general phenomena of thin-film flows 
arriving at a region where the flow is slower. For example, a similar 
phenomenon is observed in \cite{bretherton}, where the motion of air 
bubbles in a fluid that is flowing inside a tube is analysed. When the tube is 
placed horizontally and its radius is supposed to be small enough as to 
neglect gravity effects, equation (\ref{WilsonJones}), but with a power $2$ 
non-linearity, is derived to describe the flow of fluid separating the bubble 
surface from the tube wall. This flow is then matched into a minimal surface 
describing the front tip of the bubble. We point out that oscillatory behaviour for the 
equations obtained in \cite{bretherton} and in \cite{WilsonJones} have been numerically reproduced in 
\cite{TuckRev}. Another example of such oscillatory behaviour of thin-film flows in the presence of a slower flow 
ahead is provided by thin-film flows on inclined planes (cf. \cite{BertozziPhF}, \cite{BertozziEJAM}). 

We notice that general third order nonlinear ODEs of the form
\begin{equation}
\frac{d^{3}\Phi}{d\tau^{3}} = f(\Phi)\,, \label{generalODE} 
\end{equation} 
arise in many other problems of fluid mechanics and in the present context of thin-film approximations in particular 
(cf. \cite{Eggers}, \cite{EggersII}, \cite{MyersRev}, \cite{MyersSolidII}, \cite{MyersSolidI}, \cite{wilson}). 
Some of them have been studied rigorously. Typically, the type of problem 
that has been considered in the rigorous mathematical literature is the possibility 
of having solutions that vanish at a finite value of $\tau$, as $\tau$ increases. 
For instance, in \cite{BerettaHP} a necessary condition on the non-linearity $f$, such that $\Phi$ vanishes 
in the forward direction is derived. In the particular cases $f(\Phi)=1-1/\Phi^{n}$ with $n=2$ and $n=3$, 
this necessary condition is not satisfied. This implies that the solutions of 
(\ref{generalODE}) are globally defined and remain always positive for 
arbitrarily large values of $\tau$. Qualitative properties of the solutions of 
the equation (\ref{generalODE}) with a non-linearity of the form 
\[ 
f(\Phi)=\frac{v}{\Phi^{n-1}}+\frac{J}{\Phi^{n}}
\]
have been studied in \cite{Boatto}. Some specific non-linearities relevant to 
applications in fluid mechanics are described in \cite{BerettaHP}, 
\cite{Boatto} and \cite{TuckRev}. In this last paper, numerical solutions 
showing oscillatory behaviours for some class of non-linearities are obtained. 
However, to our knowledge, no rigorous result describing the oscillatory 
behaviour for the solutions of (\ref{generalODE}) have been given before.

The paper is organised as follows. In Section~\ref{section:2} we derive (\ref{capalimite}). 
We perform a thin-film approximation on the corresponding Stokes flows, that in two dimensions, 
is written in curvilinear coordinates. 
We derive the leading order problem and analyse the significance of the driving term $Q(s)$. As it was 
pointed out earlier, this coefficient is completely determined by the geometry of the substrate. 
In order to justify our assumption on $Q(s)$, we also prove in this section the existence of curves for 
which $Q(s)$ has the asymptotic behaviour $Q(s)\sim A s^2$ as $s\to 0$ for some $A\in\mathbb{R}$. 
We end the section by deriving the boundary layer problem near a point where  $Q(s)$ vanishes, that is, 
we derive equation (\ref{capalimite}) and the corresponding matching conditions. The expressions 
in curvilinear coordinates of the linear operators needed to derive the next order term in the thin-film 
approximation are given in Appendix~1.

In Section~\ref{section:3} we analyse (\ref{WilsonJones}). We start in 
Section~\ref{section:3.1} with the local behaviour near the critical point and the behaviour of 
solutions on the (one-dimensional) stable manifold, and the ones outside this manifold. For the 
former ones we prove that either $\Phi\to \infty$ as $\tau \to -\infty$ or $\Phi\to 0^+$ as 
$\tau\to (\tau^*)^+$ for a finite $\tau^*$. For the later ones, we prove rigorously that there 
exist a two dimensional family of positive solutions of (\ref{WilsonJones}), defined for 
$-\infty<\tau<\infty$, satisfying $\lim_{\tau\to -\infty} \Phi( \tau)=1$, 
$\lim\sup_{\tau\to\infty}\Phi(\tau)=\infty$ and $\lim\inf_{\tau\to\infty}\Phi(\tau)=0$. 
These solutions thus oscillate with increasing amplitude, a fact observed and described in 
\cite{WilsonJones}. In Section~\ref{section:3.2} we make precise and prove rigorously 
the asymptotic formulae obtained in \cite{WilsonJones} describing the rate of increase 
of the oscillations. Moreover, we prove that this precise type of oscillatory behaviour is an attractor 
for the solutions of (\ref{WilsonJones}), in the sense that every solution of this equation that does not approach 
$\Phi=1$ as $\tau\to\infty$, eventually oscillates in that precise way. 

Finally, we remark that the oscillatory behaviour of solutions of (\ref{WilsonJones}) will be shown to arise due to the 
transition between ranges of $\tau$ in which 
(\ref{WilsonJones}) can be approximated either by the equation 
\begin{equation}\label{SEphi-inf}
\left( \frac{d^3\Phi}{d\tau^3}+ 1 \right)  = 0\,,
\end{equation} 
when $\Phi$ is large, or by the equation
\begin{equation}\label{SEphi0}
\left(\frac{d^3\Phi }{d\tau^3} \right)\Phi^3= 1\,,
\end{equation}
when  $\Phi$ is very small. Equation (\ref{SEphi-inf}) can readily be integrated. 
The analysis of (\ref{SEphi0}) is done in Appendix~2, where the relevant information 
about the solutions is gathered in a theorem. Equation (\ref{SEphi0}) is invariant under the transformation 
$(\tau,\Phi)\to (l^{\frac{4}{3}}\tau,l \Phi)$ for any $l\in\mathbb{R}$ and can thus be analysed by means of a related transformation that reduces the 
study of this third order equation to the study of a planar autonomous system of ODEs.

\section{Derivation of the model}\label{section:2}
We first derive the equations describing the motion of a fluid in a stationary
regime flowing down a substrate. Since we restrict ourselves to
two-dimensional flows, the liquid phase fills a two-dimensional domain 
$\Omega$. We assume that the fluid wets completely the surface. Then, the 
boundary of this domain can be decomposed in two disjoint pieces 
$\partial\Omega_{s}$ and $\partial\Omega_{f}$. The portion of the boundary 
$\partial\Omega_{s}$ denotes the intersection of the liquid with the fixed 
substrate. On the other hand, $\partial\Omega_{f}$ is the free interface 
separating the liquid and the gas phase. 

We let $h_{c}$ denote the characteristic height of the liquid over the 
surface, and $R_{s}$ is the typical radius of curvature of $\partial\Omega_{s}$. 
We henceforth assume that $\frac{h_{c}}{R_{s}}$ is small in most part of the flow. 
This allows us to use a thin-film approximation. In particular, this means 
that the normal component of the velocity is very slow compared the to tangential one. 

We deal with very viscous flows, thus we assume that the flow takes place at a very low Reynolds number 
($\operatorname{Re}\ll 1$) and, therefore, that inertia effects are negligible compared to viscous forces. 
Taking into account our initial assumption, the Reynolds number is $\operatorname{Re}=h_{c}^2v_{t}/(R_s\nu)$. 
As it is usual in the study of viscous flows under the thin-film approximation, the 
order of magnitude of the tangential velocity along the surface is
\begin{equation}\label{A1}  
v_t \approx \frac{g \, h_c^2}{\nu}\,,
\end{equation}
where $\nu$ is the kinematic viscosity. The total flux of fluid per unit length 
along a surface transversal to the substrate is then of order 
$J\approx\frac{gh_{c}^{3}}{\nu}$. With this, we observe that in particular 
\[
\operatorname{Re}\approx\frac{g\,h_c^4}{R_s \nu^2}\approx\frac{h_c}{R_s}\frac{J}{\nu}\,,
\]
i.e. the main assumption on the ratio of the characteristics lengths allows rather \textit{fast} very viscous flows (cf. \cite{acheson}).

Under these assumptions we can describe the flow using the free surface Stokes 
equations. In Cartesian coordinates $(x,y)$ in $2D$ we take the gravity vector 
pointing downwards and so $\mathbf{g}=-g\mathbf{e}_y$ where 
$\mathbf{e}_y=(0,1)^T$ and $g$ is the gravity constant. 

We non-dimensionalise the problem with the length $R_s$, as 
follows:
\begin{equation}\label{nondim}
\mathbf{g} = -g\mathbf{e}_y \,,\ \mathbf{x} = R_s \mathbf{x}^{\ast}\,,
\ \mathbf{u} = \frac{g \, R_s^2}{\nu}\mathbf{u}^{\ast}\,,\ p=\rho g R_{s}p^{\ast}\,,
\ K=\frac{1}{R_s} K^{\ast}\,,\ t=\frac{\nu}{g \,R_s} \,t^{\ast}\ ,
\end{equation}
where $\mathbf{u}$ is the velocity field, $p$ the fluid pressure, $t$ is time and $K$ is the curvature of the free 
surface ($\partial\Omega_f$). In these non-dimensional variables, where we drop the $*$ for simplicity of notation, 
the Stokes flow reads
\begin{equation}\label{stokes}
\Delta\mathbf{u}-\nabla p-\mathbf{e}_{y}    =0\,,
\quad \ \mbox{div}\mathbf{u}    =0\,,\quad\mbox{in}\quad\Omega\,,
\end{equation}
and is supplemented with kinematic and surface tension boundary conditions on the free surface:
\begin{equation}
\label{surface:tension2}
v_{N}   =\mathbf{u}\cdot\mathbf{N} \,,\quad \ 
(\nabla\mathbf{u}+\left(  \nabla\mathbf{u}\right)^{T})\mathbf{N}-p\mathbf{N}  
= \frac{1}{B}\,K\,\mathbf{N} \,,\quad\mbox{on}\quad\partial\Omega_{f}\,,
\end{equation}
and with the no-slip boundary condition on the substrate: 
\begin{equation}
\mathbf{u}=0\quad\mbox{on}\quad\partial\Omega_{s}\,. \label{noslip}
\end{equation}
We notice that the order of magnitude of the velocity is given by (\ref{A1}). 
Since in problems where the thin-film 
approximation is valid the main component of the velocity is the tangential 
one, using (\ref{nondim}), it then follows that the order of magnitude of the (non-dimensional)
tangential component of $\mathbf{u}$ is $\left(  \frac{h_{c}}{R_{s}}\right)^{2}\ll 1$. The 
other dimensionless parameter appearing in (\ref{stokes})-(\ref{surface:tension2}) 
is the Bond number:
\[
B=\frac{\rho g \,R_s^2}{\sigma}=\frac{\rho g\,R_s}{\sigma/R_s}\,, 
\]
that measures the relative size of gravity and surface tension or, more 
precisely, the ratio between the hydrostatic pressure due to changes of height 
of order $R_s$ and the changes of pressure induced by the surface tension on 
surfaces with curvature radius $R_s$.

In this paper we assume that $B$ is of order one. In particular, since 
$|\mathbf{u}| \ll 1$, the second equation in  
(\ref{surface:tension2}) implies that $p\approx\frac{1}{B}\,K$, i.e. 
curvature and gravity effects are of the 
same order of magnitude; that is the characteristic feature of the limit 
considered in this paper. 

\subsection{Curvilinear coordinates}\label{section:2.1}
We now proceed to derive the equations describing the height of the liquid
using the thin-film approximation. To this end we reformulate 
(\ref{stokes})-(\ref{surface:tension2}) using a more convenient set of curvilinear 
coordinates. 

Let the substrate be described by a regular curve parametrised by its 
arc-length $s$:
\[
\partial\Omega_{s}=\{\alpha(s)=(x(s),y(s)):\ s\in\mathbb{R}\}\,.
\]
We let $\mathbf{n}$ denote the outer normal vector of $\partial\Omega_{s}$ (thus when 
the substrate is horizontal and $s$ increases, if $(x(s),y(s))$ moves towards 
the right then $\mathbf{n}$ points upwards), i.e. $\mathbf{n}(s)=(-y^{\prime}(s),x^{\prime}(s))$. 
Then the curvature of $\partial\Omega_{s}$ is $k(s)=x^{\prime}(s)y^{\prime\prime}(s)-x^{\prime\prime}(s)y^{\prime}(s)$. 
We shall also denote the tangent vector to $\partial\Omega_{s}$ by $\mathbf{t}(s)=(x^{\prime}(s),y^{\prime}(s))$ 
(observe that then $\det\left(  \mathbf{t}(s),\mathbf{n}(s)\right)  =+1$). We 
introduce the coordinates $(s,d)$ such that a point in space is determined 
relative to the substrate by its distance to it:
\[
\mathbf{X}=(x(s,d),y(s,d))=\mathbf{\alpha}(s)+d\,\mathbf{n}(s)\,.
\]
The corresponding orthonormal basis is then
$\{\hat{\mathbf{e}}_{1}=\mathbf{t},\quad\hat{\mathbf{e}}_{2}=\mathbf{n}\}$, 
which only depends on $s$. We shall let the vector components in curvilinear coordinates have 
indexes $1$ (component tangential to the substrate) and $2$ (component normal 
to the substrate).

The free boundary is described by an unknown function measuring the distance 
to the substrate $d=h(s,t)$, hence it is parametrised by 
\[
\mathbf{X}_{f}(s)=\alpha(s)+h(s)\mathbf{n}(s)\,, 
\]
its tangent and normal vectors are given by
\[
\mathbf{T}=T_{1}\mathbf{t}+T_{2}\mathbf{n}=\frac{(1-k\,h)\mathbf{t}
+\displaystyle{\frac{\partial h}{\partial s}}\mathbf{n}}{
 \left((1-k\,h)^{2}+\displaystyle{\left(  \frac{\partial h}{\partial s}\right)^{2}}\right)^{1/2}
}\,, 
\quad  
\mathbf{N}=N_{1}\mathbf{t}+N_{2}\mathbf{n}=\frac{
-\displaystyle{ \frac{\partial h}{\partial s}}\mathbf{t}+(1-k\,h)\mathbf{n}}{
\left(  (1-k\,h)^{2}+\displaystyle{\left(  \frac{\partial h}{\partial s}\right)  ^{2}}\right)^{1/2}
}\,.
\]

The curvature of the free boundary is
\begin{equation}
K = \frac{k+\displaystyle{\frac{\partial^{2}h}{\partial s^{2}}}-2k^{2} h
+ k^{3}h^{2} - k\,h\displaystyle{\frac{\partial^{2}h}{\partial s^{2}}} 
+ 2k\displaystyle{\left(  \frac{\partial h}{\partial s}\right)^{2}}
+ \frac{dk}{ds}h\displaystyle{\frac{\partial h}{\partial s}}}{
\left((1-k\,h)^{2} + \displaystyle{\left(  \frac{\partial h}{\partial s}\right)^{2}}\right)^{3/2}
}\,. 
\label{curv:fb}
\end{equation}

A transformation of the relevant differential operators to the new coordinates 
is given in e.g. \cite{acheson}. Using this, problem (\ref{stokes})-(\ref{noslip}) in 
the new coordinates becomes
\begin{equation}
\frac{1}{1-k\,d}\left(  \frac{\partial u_{1}}{\partial s}-k\,u_{2}\right)
+\frac{\partial u_{2}}{\partial d}=0\quad\mbox{in}\quad0<d<h\quad
\label{mass:curvi}
\end{equation}
and
\begin{eqnarray}
\frac{k^{\prime}d}{(1-k\,d)^{3}} \left( \frac{\partial u_{1}}{\partial s} -
k\, u_{2}\right) + \frac{1}{(1-k\, d)^{2}} \left( \frac{\partial^{2} u_{1}}{\partial s^{2}}
 -k^{2} u_{1}-2k\frac{\partial u_{2}}{\partial s} -k^{\prime}\,u_{2}\right)  \notag \\
-\frac{k}{1-k\, d}\frac{\partial u_{1}}{\partial d} +\frac{\partial^{2} u_{1}}{\partial d^{2}} 
-\frac{1}{1-k\,d}\frac{\partial p}{\partial s} - \mathbf{e}_y\cdot\mathbf{t}=0\,, 
\label{curv1}\\
\frac{k^{\prime}d}{(1-k\,d)^{3}} \left( \frac{\partial u_{2}}{\partial s} +
k\, u_{1}\right) + \frac{1}{(1-k\, d)^{2}} \left( \frac{\partial^{2} u_{2}}{\partial s^{2}} 
-k^{2} u_{2}+2k\frac{\partial u_{1}}{\partial s} +k^{\prime}\,u_{1}\right)  \notag \\
-\frac{k}{1-k\, d}\frac{\partial u_{2}}{\partial d} +\frac{\partial^{2} u_{2}}{\partial d^{2}} 
-\frac{\partial p}{\partial d}- \mathbf{e}_y\cdot\mathbf{n}=0\,, \label{curve2}
\end{eqnarray}
with
\begin{align*}
((\nabla\mathbf{u}+\nabla\mathbf{u}^{T})\mathbf{N})\cdot\mathbf{N}-p  &
=\frac{1}{B}K\quad\mbox{on}\quad d=h \\
((\nabla\mathbf{u}+\nabla\mathbf{u}^{T})\mathbf{N})\cdot\mathbf{T}  &
=0\quad\mbox{on}\quad d = h 
\end{align*} 
where the full expressions of the tangential and normal components of 
$(\nabla\mathbf{v}+\nabla\mathbf{v}^{T})\mathbf{N}$ are given in Appendix~1. 
Finally, the second equation in (\ref{surface:tension2}) becomes
\begin{equation}\label{kine:curv}
\frac{\partial h}{\partial t}=u_{2}-\frac{u_{1}}{1-k\,d}
\frac{\partial h}{\partial s}\quad\mbox{on}\quad d=h\, . 
\end{equation}

We also observe that using this coordinate system, the flux $J$ per unit 
length along any surface orthogonal to the substrate, can be written in the 
original dimensional set of variables $(\mathbf{u,x})$ (see (\ref{nondim})) as:
\begin{equation}\label{fluxscale} 
J=\frac{g R_{s}^{3}}{\nu}\int_{0}^{h} u_{1}\,dy\,. 
\end{equation}
Obviously, $\int_{0}^{h}u_{1} \, dy$ is just the non-dimensional flux.

\subsection{Thin-film approximation}\label{section:2.2}
We will assume that in most of the fluid we have:
\begin{equation}\label{h0}
\varepsilon:=\frac{h_{c}}{R_{s}}\ll1\,. 
\end{equation}
Combining (\ref{A1}) and (\ref{nondim}) it then follows that, in the non-dimensional variables, 
$u_{1}\approx\varepsilon^{2}\ll 1$, therefore, also that 
$\int_{0}^{h}u_{1}dy\approx\varepsilon^{3}\ll 1$ (and this in particular means, 
in view of (\ref{fluxscale}), that $\frac{J\nu}{gR_{s}^{3}}\approx\varepsilon^{3}\ll 1$). 
We the rescale the variables as follows:
\begin{equation}\label{thin:scale}
d=\varepsilon d^{\ast}\,,\ h=\varepsilon h^{\ast}\,,\ t=\varepsilon t^{\ast}\,,
\ u_{1}=\varepsilon^{2}u_{1}^{\ast}\,,
\ u_{2}=\varepsilon^{3}u_{2}^{\ast}\,, 
\end{equation}
where the last one results, to leading order as $\varepsilon\to 0$, from 
the rescaling of (\ref{mass:curvi}). With these changes, (\ref{mass:curvi})-(\ref{curve2}) reduce, to leading order as 
$\varepsilon\to 0$, to
\begin{equation}\label{lo}
\frac{\partial u_{1}^{\ast}}{\partial s}+
\frac{\partial u_{2}^{\ast}}{\partial d^{\ast}}    =0\,, \quad 
\frac{\partial^{2}u_{1}^{\ast}}{\partial\left(  d^{\ast}\right)  ^{2}}
-\frac{\partial p}{\partial s}-\mathbf{e}_{y}\cdot\mathbf{t}   =0\,, \quad 
-\frac{\partial p}{\partial d^{\ast}}    =0 
\end{equation}
with boundary conditions
\begin{equation}\label{tan:norm:lo}
\frac{\partial u_1^{\ast}}{\partial d^{\ast}}=0\,,\ -p=\frac{1}{B}
K\,,\quad\mbox{on}\quad d^{\ast}=h^{\ast}
\end{equation}
where, from (\ref{curv:fb}),
\begin{equation}\label{curv:expan}
K = k + \varepsilon\left(  \frac{\partial^2 h^{\ast}}{\partial s^2}
+k^2 h^{\ast}\right)  + O(\varepsilon^2)\,. 
\end{equation}
The no-slip condition (\ref{noslip}) stays as is, and (\ref{kine:curv}) becomes
\begin{equation}\label{kine:lo}
\frac{\partial h^{\ast}}{\partial t^{\ast}} = u_2^{\ast}-u_1^{\ast}
\frac{\partial h^{\ast}}{\partial s} 
\end{equation}
to leading order. 

These equations are now combined into a single one for $h$. Integrating the second equation in (\ref{lo}) one gets
\[
u_1^{\ast}=\frac{1}{2}\left(\frac{\partial p}{\partial s}
+\mathbf{e}_y\cdot\mathbf{t}\right)  ((d^{\ast})^2 - 2 h^{\ast}d^{\ast})\,,
\]
and, from the first equation in (\ref{lo}), one gets  
$u_{2}^{\ast}=-\int_{0}^{h^{\ast}}\frac{\partial u_{1}^{\ast}}{\partial s}\,dy$, 
and so
\begin{equation}\label{thin-film}
\frac{\partial h^{\ast}}{\partial t^{\ast}}=-\frac{\partial}{\partial s}
\int_{0}^{h^{\ast}}u_{1}^{\ast}\,dy=\frac{1}{3}\frac{\partial}{\partial s}
\left(  \left(  \frac{\partial p}{\partial s}+\mathbf{e}_{y}\cdot
\mathbf{t}\right)  \left(  h^{\ast}\right)  ^{3}\right)  \,. 
\end{equation}

Equations similar to (\ref{tan:norm:lo}), (\ref{curv:expan}) and (\ref{thin-film}), 
where the main driving terms are the gravity and the curvature of the 
substrate, have been obtained, in a slightly different context, in \cite{MyersSolidI} and 
\cite{MyersSolidII}. This model can be obtained also as a 
particular case of the ones considered in \cite{Royetal} for specific choices 
of the parameters. 

We shall write the leading order of (\ref{thin-film}) as
\begin{equation}\label{givenQ}
\frac{\partial h^{\ast}}{\partial t^{\ast}}+\frac{1}{3}\frac{\partial}{\partial s}
\left(  Q(s)  \left(  h^{\ast}\right)  ^{3}\right)  =0\,,\quad  \mbox{with} \quad
Q(s):=\displaystyle{\left(  \frac{1}{B}\frac{dk}
{ds}-\mathbf{e}_{y}\cdot\mathbf{t}\right)  }\,.
\end{equation}
This is a non-linear non-homogeneous (in $s$) hyperbolic equation. The term $Q\left(  s\right)$
measures the tendency of the fluid to move in the direction of the tangent 
vector $\mathbf{t}$. It is important to remark that the function $Q\left( s\right)$ 
is completely determined by the geometry of the curve and the 
gravitational field. Moreover, all the geometrical features of the curve are 
reduced in this approximation to the function $Q(s)$.

We are interested in the study of this system in the steady state regime. Which, to the current order of approximation, 
is simply
\begin{equation}\label{S1E3}
\frac{1}{3} Q(s)\left(  h^{\ast}\right)^3=1\,.
\end{equation}
We observe that (\ref{S1E3}) can have solutions with $Q(s)<0$, $h^{\ast}( s)<0$. 
The meaning of negative values of $h^{\ast}$ is just that the fluid is placed in the 
direction of the vector $-\mathbf{n}(s)$, i.e. if a point moves 
along the curve $\partial\Omega_{f}$ in the direction of increasing $s$,  
positive values of $h^{\ast}$, indicate that the fluid would be seen by the 
moving particle to the left , and negative values of $h^{\ast}$ mean that 
the particle would see the fluid to its right. We can always assume that 
$Q(s)\geq 0$ by reversing the direction of the parametrisation. 

We consider the case in which $Q(s)$ vanishes for some values of $s$. 
At these points the approximation (\ref{S1E3}) would break down, because the height of the fluid is expected to 
increase unboundedly; the next order terms (in particular, in (\ref{curv:expan})) will become important near these points. 
We will not consider in this paper curves for which $Q(s)$ takes negative values over some intervals.

\subsection{Geometrical problem of the substrate}\label{section:2.3}
We are interested in studying flows for which the approximation (\ref{S1E3}) 
breaks down; that is, if $Q(s_{0})  =0$ for some value of 
$s_{0}$, where, for definiteness, $Q(s)>0$ for $s\neq s_{0}$ in a neighbourhood of $s_0$. 
It is then relevant to show that there exist surfaces where these conditions are 
satisfied. We prove in this Section that this is indeed the case.

We begin formulating the problem satisfied by curves if 
$Q(s)$ is a given function. We recall that, since $\alpha(s)=(x(s),y(s))$ is the arc-length 
parametrisation of $\partial\Omega_{s}$, $(x'(s))^{2}+ (y'(s))^{2}=1$. 
Let us denote by $\theta(s)$ the angle between the tangent vector $\mathbf{t}$ 
and the horizontal axis. Then:
\begin{equation}\label{tanvec}
\frac{dx}{ds}=\cos(\theta(s))\,,\quad\frac{dy}{ds}=\sin(\theta(s))\,.
\end{equation}
Observe that (restricting to values of $\theta\in[-3\pi/2,\pi/2]$) when $\theta\in(-3\pi/2,-\pi/2)$ the fluid is below 
the substrate, and, when $\theta\in(-\pi/2,\pi/2)$, then fluid is on top of the substrate. For $\theta=-\pi/2$ the surface 
is vertical and the fluid is to the right of it. Similarly, for $\theta=\pi/2$ (or $-3\pi/2$), the 
substrate is vertical and the fluid is to the left of it. In terms of $\theta$, 
the curvature is given by $k=\frac{d\theta}{ds}$ and $\mathbf{e}_{y}\cdot\mathbf{t}=
\sin\left(\theta(s)\right)$. We can now write (\ref{givenQ}) as 
\begin{equation}\label{forced:pendulum}  
\frac{1}{B}\frac{d^{2}\theta}{ds^{2}}-\sin\theta(s)=Q(s)\,.
\end{equation}
This is the differential equation that describes the geometry of the substrate 
in terms of the function $Q(s)$. Solving (\ref{forced:pendulum}) 
the curve is recovered by integrating (\ref{tanvec}). Observe that 
(\ref{forced:pendulum}) is the equation of a forced nonlinear pendulum. 

There is a class of semi-explicit solutions of (\ref{forced:pendulum}) that 
can be obtained assuming that $Q\left(  s\right)  =Q_{0}\in\mathbb{R}$. The 
simplest case corresponds to $\left\vert Q_{0}\right\vert \leq1$, and includes 
the pathological case $Q_0\equiv 0$, for which the thin-film approximation breaks 
down everywhere on the substrate (see (\ref{S1E3})). For such $Q_0$'s there 
are constant solutions of (\ref{forced:pendulum}) given by
\begin{equation}\label{B1}
\theta(s)  =\theta_0\,,\quad \sin(\theta_0)=-Q_0\,. 
\end{equation}
These solutions correspond to an inclined plane with constant slope 
$\theta_{0}$. We can always assume without loss of generality that 
$Q_{0}\geq0$. We will assume also that $\theta_{0}\in\left[  -\pi,0\right]$. Due 
to our choice of signs, the fluid is above the plane 
$\left\{  \theta\left( s\right)  =\theta_{0}\right\}$ if $\theta_{0}\in\left(  -\frac{\pi}{2},0\right]$ 
and the fluid is below the same plane if 
$\theta_{0}\in\left[-\pi,-\frac{\pi}{2}\right)$. In the critical case $\theta_{0}=-\frac{\pi}{2}$ 
the plane is vertical and the fluid is to the right of it. Also, 
for any $0\leq Q_{0}<1$ there are two roots of the second equation of 
(\ref{B1}) with $\theta_{0}\in\left[  -\pi,0\right]$. Let us them be denoted 
by $\theta_{1}$ and $\theta_{1}^{\ast}$, where
\[
-\pi\leq\theta_{1}^{\ast}<-\frac{\pi}{2}<\theta_{1}\leq0\,,  
\]
and they satisfy $\theta_{1}^{\ast}+\theta_{1}=-\pi$.

We can also consider perturbations of these constant solutions. The point 
$\left(  \theta,\frac{d\theta}{ds}\right)  =\left(  \theta_{1}^{\ast},0\right)$ 
is a centre in the phase portrait of the equation
(\ref{forced:pendulum}). Therefore, there is a one-parameter family of 
periodic solutions of (\ref{forced:pendulum}) for which $\theta\left(s\right)$ 
oscillates around $\theta_{1}^{\ast}$ in the usual way for the 
nonlinear pendulum. We can then obtain curves $\partial\Omega_{s}$ by means of 
(\ref{tanvec}). The resulting curve does not have self-intersections if the 
amplitude of the oscillations of $\left(  \theta\left(  s\right)  -\theta_{1}^{\ast}\right)$ 
is small, as it can be seen using a continuity argument. 
It is interesting to notice that for all these surfaces the height of the 
liquid remains constant (see (\ref{S1E3})), even if they are not planar. 
The fluid is below them, but we would assume that the fluid is stable enough to 
allow such flows.

We remark that the quantity $H\left(\frac{d\theta}{ds},\theta\right)=
\frac{1}{2B}\left(\frac{d\theta}{ds}\right)^{2}+\cos(\theta)+\sin(\theta_0)  \theta$ 
is conserved under the flow (\ref{forced:pendulum}).
Using this, it can be seen that there is a homoclinic orbit 
connecting the point $\left(  \theta,\frac{d\theta}{ds}\right)  =\left(\theta_{1},0\right)$ 
to itself. In view of  (\ref{S1E3}), this provides another example of a surface 
on which the height of the fluid remains constant. 
A perturbative argument shows that the corresponding surface obtained by means of 
(\ref{tanvec}) approaches asymptotically, as $s\to\pm\infty$, the 
plane $\{\theta(s) \equiv\theta_1 \}$ and it does not 
have self-intersections if $\theta_0$ is close to $-\frac{\pi}{2}$. However, 
since this is not the main goal of this paper, we will not continue this 
discussion here.

A set of interesting solutions of (\ref{forced:pendulum}) are the ones 
associated to $\left\vert Q_{0}\right\vert >1$. For these solutions
$\left\vert \theta(s)\right\vert$ increases very fast as 
$s\to\infty$. Their asymptotics are $\theta(s)\sim\frac{B Q_0 s^2}{2}$ as $s\to\pm\infty$. 
The corresponding surface obtained by means of (\ref{tanvec}) approaches asymptotically a finite 
point in each of the limits $s\to\infty$ and $s\to -\infty$. Therefore, these surfaces cannot be extended to unbounded domains 
and the approximations yielding to (\ref{S1E3}) must breakdown at some 
point. However, these curves give interesting examples of substrates where the driving 
coefficient $Q(s)$ is constant and larger than the maximum value allowed by the  
gravitational force, due to curvature effects.

After these preliminary observations about (\ref{forced:pendulum}), in the following theorem 
we construct functions $Q(s)$ that vanish quadratically at $s_0=0$.
\begin{theorem}\label{ExtCurves}
Let $0<Q_0 <1$. There exist $\omega_0>0$ small, such that 
for any $\omega<\omega_0$, there exist curves $\partial\Omega_s\in C^{\infty}(\mathbb{R})$ 
without self-intersections that approach asymptotically, as $s\to\pm\infty$, the line 
$\{  \theta(s)\equiv \theta_1 \}$, where  
$\theta_1\in\left(-\frac{\pi}{2},0\right)$ and $\sin(\theta_1)=-Q_0$, 
and such that the function $Q(s)$ defined by means of (\ref{forced:pendulum}) 
satisfies $Q(s)>0$ for $s\neq 0$, $Q(0)=\omega$ and $0<\lim_{s\to 0}\frac{Q(s)-\omega}{s^2}<\infty$.
\end{theorem}

\begin{proof}
We construct the function $Q(\cdot)$ as follows. We take 
$Q(s)  =Q_0 \left(1-\xi\left( \frac{s}{\omega}\right)\right)$ 
for $s\leq\omega,$ where $\xi\in C^{\infty}(\mathbb{R})$, 
$\xi^{\prime}(z)>0$ for $z\in\left(0,\frac{1}{2}\right)$, 
$\xi^{\prime}(z)<0$ for $z\in\left(\frac{1}{2},1\right)$, 
$\xi(z)  =0$ for $z\in (-\infty,0]\cup[1,\infty)$, 
$\xi^{\prime\prime}\left(\frac{1}{2}\right)<0$ and $\int_0^1 \xi(z)dz=1$. 
The solution $\theta_\omega$ of the equation (\ref{forced:pendulum}) can then be approximated 
for $s\leq\omega$, if $\omega$ is small enough by means of the solution of the distributional 
equation:
\[
\frac{1}{B}\frac{d^2\theta}{d s^2}-\sin\theta(s)=Q_0 -  Q_0\omega \delta(s)  \,.
\]
An alternative way of proving this can be obtained by rescaling $s$ with $\omega$ and studying 
the resulting regular perturbation problem using Gronwall. We then obtain:
\[
\lim_{\omega\to 0}\frac{\left| \theta_{\omega}(s)-\theta_1 \right|}{\omega}=0\,,
\quad \lim_{\omega\to 0}\frac{\left|\theta_{\omega}^{\prime}(s)  - B Q_0 \omega\right| }{\omega}=0 
\]

We now construct $\theta_{\omega}(s)$, but not $Q(s)$, for $s\geq\omega$ as any function that extends the obtained 
$\theta_{\omega}(s)$ in the region $s<\omega$ as a function in $C^{\infty}(\mathbb{R})$ 
with $| \theta_{\omega}(s)-\theta_1 | \leq\frac{\omega}{4}$, 
$|\theta_{\omega}^{\prime}(s)| \leq 2 B Q_0\omega$, $|\theta_{\omega}^{\prime\prime}(s)| \leq 2B\omega$ for 
$s\geq\omega$. We impose also that $\lim_{s\to\infty}\theta_{\omega}(s)=\theta_1$. 
This extension is possible because $|\theta_{\omega}^{\prime\prime}(s)| \ll\omega$ 
as $\omega\to 0$. We then define $Q(s)$ by means of (\ref{forced:pendulum}) for $s\geq\omega$ and 
construct the curve $\partial\Omega_s$ by means of (\ref{tanvec}). Since $|\theta_{\omega}(s)-\theta_1|
 \leq\frac{\omega}{4}$ for $s\in\mathbb{R}$ it follows that the resulting curves do not 
self-intersect for $\omega$ small enough. The rest of the properties stated in 
the Theorem are straightforward. 
\end{proof}

We recall that our sign criteria implies that the fluid lies above the curve 
$\partial\Omega_s$ obtained in Theorem~\ref{ExtCurves}.

\subsection{Analysis of water accumulation regions in the limit $\varepsilon\rightarrow 0$}\label{section:2.4} 
We now study the form of the stationary solutions of the thin-film approximation 
in the regions where the leading order approximation (\ref{S1E3}) breaks down. We 
will consider only the case in which $Q^{\prime\prime}(s_0)>0$ 
at the points $s=s_{0}$ where $Q(s)$ vanishes. The existence of such non self-intersecting curves 
$\partial\Omega_{s}$ with $Q\in C^{\infty}\left(  \mathbb{R}\right)$ 
satisfying (\ref{S3E2}) below has been proved in Theorem~\ref{ExtCurves}. 
Assuming (\ref{h0}) and performing the scalings (\ref{thin:scale}) into 
the Stokes problem (\ref{curv1})-(\ref{kine:curv}), one can derive the stationary thin-film approximation
 to order $\varepsilon$, namely, dropping the $*$'s for simplicity of notation,
\begin{equation}\label{S5E1}
\frac{1}{3}\frac{\partial}{\partial s}\left(  \left(  Q(s) +
\varepsilon \frac{\partial^3 h}{\partial s^3 }+\varepsilon k^2 \,\frac{\partial h}{\partial s} 
-\varepsilon(\mathbf{e}_y\cdot\mathbf{n}) \frac{\partial h}{\partial s}\right)  h^3 
+ \varepsilon\tilde{Q}(s)h^4\right)  = 0
\end{equation}
with
\[
\tilde{Q}(s)=\frac{d}{ds} k^2-\frac{7}{8} k \,Q(s)+\frac{5}{8}\frac{1}{B}\frac{dk}{ds}
-\frac{3}{8}\frac{\partial}{\partial s}(\mathbf{e}_y\cdot\mathbf{n})\,.
\]
and $Q(s)$ as in (\ref{givenQ}). Assuming now that the flux is the same as the one obtained in the leading order approximation 
(\ref{S1E3}), we obtain the stationary equation: 
\begin{equation}\label{S2E2}
\frac{1}{3}\left(  Q(s) +\varepsilon\frac{\partial^{3} h}{\partial s^{3}} 
+\varepsilon k^{2}\,\frac{\partial h}{\partial s} 
-\varepsilon(\mathbf{e}_{y}\cdot\mathbf{n}) \frac{\partial h}{\partial s} + 
\varepsilon\tilde{Q}(s)h\right)  h^{3} =1\,. 
\end{equation}

We take a somewhat general form for $Q(s)$, namely, we assume that 
near a stagnation point placed at $s=0$ we have:
\begin{equation}\label{S3E2}
Q(s)  =As^{2}+b\varepsilon^{\frac{6}{17}}+
o(s^2 +\varepsilon^{\frac{6}{17}})\,, \quad A>0\,, \quad b\in\mathbb{R}\,.
\end{equation}
The exponent $\frac{6}{17}$ results from the fact that with this particular 
rescaling the effects of this term will turn out to be of order one in the 
boundary layer. We will also assume that $Q\left(  s\right)  >0$ for 
$|s| \gg\varepsilon^{\frac{3}{17}}$.

The natural rescaling of variables required to study the singular
perturbation problem (\ref{S2E2}) is:
\begin{equation}\label{S2E5a}
s=\varepsilon^{\frac{3}{17}}\left(  \frac{3}{2A^{4}}\right)^{\frac{1}{17}}\xi\,,\quad
h=\varepsilon^{-\frac{2}{17}}\left(  \frac{3}{2A^{\frac{3}{5}}}\right)^{\frac{5}{17}}H\,. 
\end{equation}
Finally, (\ref{S2E2}) then becomes (\ref{capalimite}), to leading order, where 
\[
a=\left(  \frac{2}{3A^{\frac{9}{2}}}\right)^{\frac{2}{17}}b\,.
\]

In order to determine the boundary conditions that must be imposed to the 
solutions of (\ref{capalimite}) we must study the asymptotics of the solutions of 
(\ref{S1E3}) as $s\to 0$. Using (\ref{S3E2}), we obtain from (\ref{S1E3})
\begin{equation}\label{S2E6}
h(s)  \sim\left(  \frac{3}{A s^2}\right)^{\frac{1}{3}}
\quad \mbox{as}\quad \varepsilon^{\frac{3}{17}}\ll \left\vert s\right\vert \ll 1\,,
\end{equation}
and using (\ref{S2E5a}), we find
\begin{equation}\label{S2E7}
H\sim\frac{1}{\left\vert \xi\right\vert ^{\frac{2}{3}}} \quad \mbox{as} \quad
1  \ll \left\vert \xi\right\vert \ll \varepsilon^{-\frac{3}{17}}\,,
\end{equation}
which provides the matching conditions for the solutions of
 (\ref{capalimite}).

According to the picture emerging from the results described here it would 
follow that the stationary solutions of (\ref{S5E1}) could be approximated by 
means of the solutions of (\ref{S1E3}) (i.e. by (\ref{S2E6})) except near the stagnation points or where 
$Q(s)$ is very small. At such points, a boundary layer arises and it 
is described by means of the rescaling (\ref{S2E5a}) and the 
function $H$ that solves (\ref{capalimite}) with matching conditions (\ref{S2E7}).


\section{The solutions of (\ref{WilsonJones})}\label{section:3}
In this section we analyse (\ref{WilsonJones}) rigorously. This analysis
will be used in the proof the existence of solutions of (\ref{capalimite}) subject to 
(\ref{S2E7}) that is done in \cite{CV2}.

\subsection{Global existence, stable and unstable manifolds}\label{section:3.1}
In this Section we describe in detail the solutions of (\ref{WilsonJones}). 
We write (\ref{WilsonJones}) in the equivalent system form:
\begin{equation}\label{S3E1}
\frac{d\Phi}{d\tau}=W \,, \ \frac{dW}{d\tau}=\Psi \,,\ \frac{d\Psi}{d\tau}
=\frac{1}{\Phi^{3}}-1\,. 
\end{equation}
We observe that there is a unique critical point for (\ref{WilsonJones}), namely 
$P_s=(  \Phi,W,\Psi)=(  1,0,0)$, and is hyperbolic. 
The stable manifold of (\ref{WilsonJones}) at the point $P_s$ is tangent to the vector:
$v_1=(3^{-\frac{2}{3}},-3^{-\frac{1}{3}},1)^T$, 
and the corresponding eigenvalue is $\lambda_1:=-3^{\frac{1}{3}}$. Its 
unstable manifold is two-dimensional and tangent at the point $P_s$ to the
plane spanned by the vectors $v_2:=\left(-\frac{1}{6}3^{\frac{1}{6}},\frac{1}{6}3^{\frac{2}{3}},1\right)^T$
and $v_3:=\left(\frac{1}{6}3^{\frac{5}{6}},\frac{1}{2}3^{\frac{1}{6}},0\right)^T$. 
The corresponding eigenvalues of the linearised problem being complex 
conjugates, namely, $\lambda_{2}:=\ 3^{\frac{1}{3}}\left(  1+i\,3^{\frac{1}{2}}\right)/2$ and 
$\lambda_{3}:=3^{\frac{1}{3}}\left(  1-i\,3^{\frac{1}{2}}\right)  /2$.

After these preliminary observations, we next prove that the solutions of (\ref{WilsonJones}) 
do not develop singularities for increasing $\tau$.

\begin{lemma}[Forward Global Existence]\label{global:existence1} 
Suppose that $\Phi(\tau_0)>0$, $W(\tau_0)$, $\Psi(\tau_0)$ are arbitrary. Then, the solution of 
(\ref{S3E1}) with these initial data is defined and $\Phi(\tau)>0$ for any $\tau>\tau_0$.
\end{lemma}

\begin{proof}
The only possibility of losing global existence is when $\Phi(\tau)$ approaches 
zero at a finite value of $\tau$, because when $\Phi(\tau)\geq\delta>0$ 
a simple estimate yields, integrating 
the equation, that $(\Phi,W,\Psi)  \leq (p(z),p^{\prime}(\tau),p^{\prime\prime}(\tau))$ where 
$p(\tau)=\Phi(\tau_0)+\frac{1}{6}(1/\delta^3-1)(\tau-\tau_0)^3+\Psi(\tau_0)(\tau-\tau_0)^2/2+W(\tau_0)(\tau_0)$, 
thus the solution cannot become unbounded at a finite value of $\tau$.

First, we show that if
\begin{equation}\label{S4E2}
\lim\inf_{\tau\to\tau^{\ast}} \Phi(\tau)  = 0
\end{equation}
then 
\begin{equation}\label{S4E3}
\lim_{\tau\to\tau^{\ast}} \Phi(\tau)  = 0\,.
\end{equation} 
We prove (\ref{S4E3}) by contradiction. First, we observe that as long 
as the solution is defined
\[
\frac{d\Psi}{d\tau}=\frac{1}{\Phi^{3}}-1\geq-1
\]
holds and then $\Psi(\tau)\geq\Psi(0)-\tau$, therefore
\begin{equation}\label{S4E1}
\Psi(\tau)\geq - C\tau^{\ast}\quad\mbox{for}\quad 0\leq\tau<\tau^{\ast}\mbox{and some} \quad C>0\,. 
\end{equation}
Let us now assume that (\ref{S4E2}) and that $\lim_{\tau\to\tau^{\ast}} \Phi(\tau)$ does not exist. 
Then, for some small enough $\varepsilon_0 > 0$, 
there exists an increasing sequence $\{\tau_n\}$ 
such that $\lim_{n \to\infty} \tau_n = \tau^{\ast}$, 
$\Phi(\tau_n)  = 2 \varepsilon_0$ and $W(\tau_n)\geq 0$ for all $n\in\mathbb{N}$.  
Integrating (\ref{S3E1}) gives
\[
\Phi(\tau)  = \Phi(\tau_n)  + W (\tau_n)(\tau - \tau_n)  
+ \int_{\tau_n}^{\tau}\!\int_{\tau_n}^{\eta} \Psi(s)  ds\,d\eta \,, \quad\mbox{for}\quad \tau\geq\tau_n\,,
\]
hence, using (\ref{S4E1}), we obtain that, as long as $\Phi(\tau)\geq\varepsilon_{0}$,
\begin{align*}
\Phi\left(  \tau\right)   &  \geq\Phi\left(  \tau_{n} \right)  
+ W \left(\tau_{n} \right)  \left(  \tau- \tau_{n} \right)  
- \int_{\tau_{n}}^{\tau} \!\int_{\tau_{n}}^{\eta} K \tau^{\ast} \,ds\,d\eta\\
&  \geq2 \varepsilon_{0} - \frac{K}{2} \tau^{\ast} \left(  \tau- \tau_{n}
\right)  ^{2}, \quad\mbox{for} \quad\tau^{\ast} > \tau\geq\tau_{n} \,.
\end{align*}
It then follows that $\Phi(\tau)  \geq\varepsilon_{0}$ for 
$\tau_{n} \leq\tau\leq\tau_{n} + \sqrt{\frac{2 \varepsilon_{0}}{K \left(\tau^{\ast} \right)  }}$. 
On the other hand, due to (\ref{S4E2}), there exist 
$\left\{\tilde{\tau}_{n} \right\}$ such that 
$\lim_{n \rightarrow\infty}\tilde{\tau}_{n} = \tau^{\ast}$ and 
$\Phi\left(  \tilde{\tau}_{n} \right)\leq\frac{\varepsilon_{0}}{2}$. 
We can assume that $\tilde{\tau}_{n} > \tau_{n}$. 
Since for $n$ large enough we have 
$\tilde{\tau}_{n}\in
\left(  \tau_{n},\tau_{n} + \sqrt{ \frac{2 \varepsilon_{0}}{K (\tau^{\ast})} } \right)$,  
we obtain a contradiction, and therefore (\ref{S4E3}) must hold if (\ref{S4E2}) does.

We now prove that (\ref{S4E3}) cannot happen. In order to make this precise we 
employ the change of variables $\bar{\Phi}=(\tau^{\ast}-\tau)^{-\frac{3}{4}}\Phi$, 
$\bar{W}=(\tau^{\ast}-\tau)^{\frac{1}{4}}W$, $\bar{\Psi}=(\tau^{\ast}-\tau)^{-\frac{5}{4}}\Psi$ 
with $s=-\ln(\tau^{\ast}-\tau)$, which gives the 
system
\begin{equation}\label{zero:forward}
\frac{d\bar{\Phi}}{ds}   =\bar{W}+\frac{3}{4}\bar{\Phi}\,,\quad 
\frac{d\bar{W}}{ds}    =\bar{\Psi}-\frac{1}{4}\bar{W}\,, \quad
\frac{d\bar{\Psi}}{ds}    =-\frac{5}{4}\bar{\Psi}+\frac{1}{\bar{\Phi}^{3}}
-e^{-\frac{9}{4}s}\,.
\end{equation}
We observe that with this change of variables, when $\tau\rightarrow(\tau^{\ast})^{-}$ 
then $s\to +\infty$ and the 
inhomogeneous term in (\ref{zero:forward}) becomes very small. In the absence 
of this term, the remaining autonomous system has no critical points, 
suggesting that all trajectories are unbounded. 

Let us assume that (\ref{S4E3}) holds, then, in terms of the new variables, 
this is equivalent to 
\begin{equation}\label{trans:cond}
\lim_{s\to\infty} e^{-\frac{3}{4}s}\bar{\Phi} =0\,.
\end{equation} 
First we assume that $\lim_{s\to\infty} \bar{\Phi}\neq0$. Then, multiplying 
the third equation in (\ref{zero:forward}) by $e^{\frac{9}{4}s}$ we obtain: 
\[
e^{s} \frac{d}{ds}\left(  e^{\frac{5}{4}}\bar{\Psi}\right)  = 
\frac{1}{\left(e^{-\frac{3}{4}s}\bar{\Phi}\right)  ^{3} } -1
\]
thus for large enough $s$, (\ref{trans:cond}) implies that there exists a constant 
$C>0$ such that
\begin{equation}\label{S3Ealgo}
e^{s} \frac{d}{ds}\left(  e^{\frac{5}{4}s}\bar{\Psi}\right)  > C\,.
\end{equation}
Setting $s>s_{0}$ with $s_{0}$ large, integration of (\ref{S3Ealgo}) gives
\[ 
e^{\frac{5}{4}s}\bar{\Psi}(s) >e^{\frac{5}{4}s_0}\bar{\Psi}(s_0) 
+C\int_{s_0}^{s} e^{-y}\,dy=e^{\frac{5}{4}s_0}\bar{\Psi}(s_0)
-C(e^{-s}-e^{-s_0})\geq e^{\frac{5}{4}s_0}\bar{\Psi}(s_0) \,,
\] 
hence, taking for example $\Psi_{0}=e^{\frac{5}{4}s_{0}}\bar{\Psi}(s_{0})$ 
for any such $s_{0}$, we obtain that for $s>s_0$ there exists a 
constant (of undetermined sign) such that $\bar{\Psi}(s)\geq e^{-\frac{5}{4}s}\Psi_{0}$.  
Proceeding in a similar manner, the second equation in (\ref{zero:forward}) yields 
$\bar{W}>e^{-\frac{1}{4}s}W_0$ for large $s$, and some 
constant $W_0$. Using this now for the first equation in (\ref{zero:forward}) 
gives, for large $s_0$ and $s>s_0$,
\begin{equation}\label{phi:est}
e^{-\frac{3}{4}s}\bar{\Phi}>e^{-\frac{3}{4}s_{0}}\bar{\Phi}(s_{0})+e^{-s_{0}}W_{0} - e^{-s}W_0\,.
\end{equation} 
Since $\bar{\Phi}(s)$ does not tend to $0$ as $s\to\infty$, one can 
always choose $s_{0}$ large enough such that $\bar{\Phi}(s_{0})+e^{-\frac{1}{4}s_{0}}W_{0}\geq\delta>0$, 
and then (\ref{phi:est}) implies that $\lim_{s\to\infty}e^{-\frac{3}{4}s}\bar{\Phi}>0$, 
a contradiction. 

If $\lim_{s\to\infty}\bar{\Phi}=0$, we argue in a similar way directly on 
the equations (\ref{zero:forward}). The third and the second equations give 
that there exist a large positive constant $C$ such that $\bar{\Psi}$, 
$\bar{W}>C$ for $s$ large enough. Then the first equation in 
(\ref{zero:forward}) and the fact that $\bar{\Phi}(s)>0$ or all $s$ imply 
that $\frac{d}{ds}\bar{\Phi}> C$, a contradiction. 
\end{proof}

The next lemma deals with the solutions contained on the stable manifold of the critical point.

\begin{lemma}[Behaviour of solutions on the stable manifold]\label{bh:stable:manifold} 
If $(\Phi,W,\Psi)\to P_s$ as $\tau\to\infty$ and $(\Phi,W,\Psi)\neq P_s$, then either
\begin{equation}\label{C1E1}
\Phi\rightarrow+\infty\quad\mbox{as}\quad\tau\rightarrow-\infty
\end{equation}
or there exists a finite $\tau^{\ast}$ such that
\begin{equation}\label{C1E2}
\Phi(\tau)\to 0\quad\mbox{as}\quad\tau\to(\tau^{\ast})^{+}
\quad\mbox{with}\quad\Phi(\tau)>0\quad\mbox{for all}\quad\tau > \tau^{\ast}\,.
\end{equation}
\end{lemma}
\begin{proof}
We first observe that integrating (\ref{WilsonJones}) three times we obtain
\begin{equation}\label{simple:int2}
\frac{d^{2}\Phi(\tau)}{d\tau^{2}}   =
-\int_{\tau}^{+\infty}\left( \frac{1}{(\Phi(\tau_{1}))^{3}}-1\right)  \,d\tau_{1}\,,
 \quad
\frac{d\Phi(\tau)}{d\tau} =
\int_{\tau}^{+\infty}\!\int_{\tau_2}^{+\infty}\left(\frac{1}{(\Phi(\tau_1))^3}-1\right)  \,d\tau_1\,d\tau_2 
\end{equation}
\begin{equation}\label{simple:int}
\Phi(\tau)-1=\int_{\tau}^{+\infty}\!\int_{\tau_{3}}^{+\infty}\!\int_{\tau_{2}}^{+\infty}
\left(  1-\frac{1}{(\Phi(\tau_{1}))^{3}}\right)  \ d\tau_1
\,d\tau_{2}\,d\tau_{3}\,. 
\end{equation}

Due to the fact that the trajectory under consideration is in the stable 
manifold of $P_{s}$ the corresponding orbit of $(\Phi,W,\Psi)$ is tangent to 
$P_{s}$ in the direction $v_{1}$. Suppose that this happens in such a way that 
there exists a $\bar{\tau}$ large enough such that $\Phi(\tau)>1$ for all 
$\tau\in(\bar{\tau},+\infty)$. Let $\bar{\tau}$ be
$
\bar{\tau}=\inf\{\tau:\ \Phi(\sigma)>1\ \forall\ \sigma\in(\tau,+\infty)\}
$.

Our goal is to prove that $\bar{\tau}=-\infty$. Suppose that $\bar{\tau}>-\infty$. 
Then, the continuity of $\Phi$ implies that $\Phi\left(  \bar{\tau}\right)  =1$. 
Using (\ref{simple:int}) with $\tau=\bar{\tau}$ and the definition of $\bar{\tau}$ 
we obtain a contradiction, whence $\bar{\tau}=-\infty$. Then (\ref{simple:int2}) implies 
$\frac{d\Phi(\tau)}{d\tau}\leq C<0$ for any $\tau$ sufficiently small, 
and therefore (\ref{C1E1}) follows.

Suppose now that the orbit $(\Phi,W,\Psi)$ is tangent to $P_{s}$ in the 
direction $-v_1$, then there exists $\bar{\tau}$ such that $\Phi(\tau)<1$ 
for all $\tau\in(\bar{\tau},+\infty)$. Now (\ref{simple:int2}) implies that 
$\frac{d\Phi}{d\tau}>0$ for all $\tau\in [\bar{\tau},+\infty)$. 
Thus, trivially, $\Phi(\tau)<1$ for all $\tau\in(\tau^{\ast},+\infty)$, where 
$\tau^{\ast}$ is defined by 
$\tau^{\ast}=\inf\{\tau:\ \Phi(\sigma)>0\ \forall\ \sigma\in(\tau,+\infty)\}$.
Let us prove that $\tau^{\ast}>-\infty$, so that in particular, by continuity, 
$\Phi(\tau^{\ast})=0$ and $\Phi(\tau)\in(0,1)$ for all $\tau\in(\tau^{\ast},+\infty)$. 
If, to the contrary, $\tau^{\ast}=-\infty$ then, since $\frac{d\Phi}{d\tau}>0$ and 
$\frac{d^{2}\Phi}{d\tau^{2}}<0$, it follows that
$\Phi(\tau) \leq \Phi(\tau_{1}) + \Phi^{\prime}(\tau_{1})(\tau-\tau_{1})$ for any 
$\tau_{1}\in\mathbb{R}$, hence $\tau^{\ast}>-\infty$.
\end{proof}

The solutions on the stable manifold are all those solutions with $\Phi(\tau)\to 1$ as $\tau\to \infty$: 
\begin{lemma} 
If $\Phi(\tau)$ is a solution of (\ref{WilsonJones}) with $\lim_{\tau\to\infty}\Phi(\tau)=1$, then $(\Phi,W,\Psi)\to P_s$.
\end{lemma}
\begin{proof}
We define $P(  \tau;\tau_1)$ as:
\begin{equation}
P(\tau;\tau_1) =\Phi(\tau) 
-\int_{\tau_1}^{\tau} \! \int_{\tau_1}^{\eta_1} \! \int_{\tau_1}^{\eta_2} 
\left[  \frac{1}{\Phi\left(  \eta_{3}\right)  ^{3}}-1\right]d\eta_3 \,d\eta_2 \,d\eta_1
\label{E1}
\end{equation}
Then, using $\lim_{\tau\to\infty}\Phi(\tau)=1$ 
as well as (\ref{S3E1}) it follows that for any $\tau_1$, $P(\tau;\tau_1)$ is a second 
order polynomial given by
\begin{equation}\label{E2}
P(\tau;\tau_1)=\Phi(\tau_1)+\frac{d\Phi(\tau_1)}{d\tau}(\tau-\tau_1)
+\frac{d^2\Phi(\tau_1)}{d\tau^2}(\tau-\tau_1)^2\,.
\end{equation}
Taking $\tau_1\to\infty$ and using (\ref{E1}), it is clear that
\[
\lim_{\tau_1\to\infty}P(\tau_1;\tau_1)  =\lim_{\tau_1\to\infty}P(  \tau_1+1;\tau_1) 
 =\lim_{\tau_1\to\infty} P(\tau_1+2;\tau_1)  =1\,.
\]  
Therefore (\ref{E2}) implies:
\[
\lim_{\tau_1\to\infty}\left[  \frac{d\Phi(\tau_1)}{d\tau}
+\frac{d^2\Phi(\tau_1)}{d\tau^2}\right]  =0\,,\quad  \lim_{\tau_1\to\infty}
\left[  \frac{d\Phi(\tau_1)}{d\tau}+2\frac{d^2\Phi(\tau_1)}{d\tau^2}\right]  =0 \,, 
\] 
and hence $(\Phi,W,\Psi)\to P_s$ as $\tau\to\infty$.  
\end{proof}

We now give the basic behaviour of the solutions that are not in the stable manifold:

\begin{lemma}\label{bh:unstable:manifold} 
Suppose that $\Phi$ is a solution of (\ref{WilsonJones}) such that $\Phi( \tau)$ does not 
converge to $1$ as $\tau\to\infty$. Then:
\begin{equation}\label{C1E3}
\lim\sup_{\tau\to\infty}\left(  \frac{1}{2\Phi^{2}}+\Phi\right)=+\infty\,.
\end{equation}
Moreover,
\begin{equation}\label{C1E3a}
\lim\sup_{\tau\to\infty}\Phi(\tau)  = +\infty
\quad\mbox{and}\quad 
\lim\inf_{\tau\to\infty}\Phi(  \tau)=0\,. 
\end{equation}
\end{lemma}

\begin{proof}
There is a monotonicity property associated to the solutions of (\ref{S3E1}). 
Indeed, multiplying (\ref{WilsonJones}) by $\frac{d\Phi}{d\tau}$ we obtain: 
\[ 
\frac{d}{d\tau}\left(  \frac{d^2\Phi}{d\tau^2}\frac{d\Phi}{d\tau}+\frac{1}{2\Phi^2}+\Phi\right)  
=\left(  \frac{d^2\Phi}{d\tau^{2}}\right)^2 \geq 0
\]
hence
\begin{equation}\label{S3E3}
\frac{dE}{d\tau}=\Psi^2 \geq 0 \,,\quad E:=\Psi W+\frac{1}{2\Phi^{2}}+\Phi\,.
\end{equation}

The energy estimate (\ref{S3E3}) yields estimates for $\Psi^{2}$ for the 
trajectories contained in the unstable manifold. More precisely we have: 
\[
\int_{\tau_0}^{\infty}(\Psi(s))^2 ds = E(\infty)  - E(\tau_0) 
\]
for any $\tau_0 \in \mathbb{R}$. Notice that $E(\infty)$ could 
be infinite. Actually, our goal is to show that $E(\infty)=\infty$ unless 
$\Phi\to 1$ as $\tau\to\infty$.

Suppose that $E(\infty)<\infty$, then
\begin{equation}\label{S3E4}
\int_{\cdot}^{\infty} \left(\Psi(s)\right)^{2} ds <\infty\,. 
\end{equation}
We claim that if this is the case then $\lim_{\tau\to\infty}\Psi(\tau)=0$. 
Indeed, the third equation in (\ref{S3E1}) yields:
\begin{equation}\label{S3E5}
\frac{d\Psi}{d\tau}\geq-1 \quad\mbox{for all}\quad\tau\in\mathbb{R}\,.
\end{equation}
Suppose that there exists a sequence $\{\tau_n \}$ such that 
$\lim_{n\to\infty}\tau_n=\infty$ and satisfying $\Psi(\tau_n) \geq\varepsilon_0>0$. 
Then (\ref{S3E5}) implies $\Psi(\tau)\geq\varepsilon_{0} - (\tau-\tau_n)$ 
for all $\tau>\tau_n$, hence
\[
\int_{\tau_n}^{\tau_n + \frac{\varepsilon_0}{2} } \left(\Psi(\tau)\right)^2 d\tau
\geq\frac{\varepsilon_0^3}{8}
\]
and this gives a contradiction with (\ref{S3E4}). Similarly, suppose that 
there exists a sequence $\{\tau_n\}$ such that $\lim_{n\to\infty}\tau_n=\infty$ 
and satisfying $\Psi(\tau_n)\leq -\varepsilon_0<0$. Then, using again (\ref{S3E5}) 
we obtain $\Psi(\tau) \leq -\varepsilon_0 +(\tau_n-\tau)$, $\tau<\tau_n$, hence
\[
\int_{\tau_n}^{\tau_n + \frac{\varepsilon_0}{2}} \left(\Psi(\tau)\right)^2 d\tau
\geq\frac{\varepsilon_0^3}{8}\,,
\]
that also yields a contradiction, therefore
\begin{equation}\label{S3E6}
\lim_{\tau\to\infty}\Psi(\tau)  =0 \,. 
\end{equation}

The first two equations of (\ref{S3E1}) imply:
\begin{align}
\sup_{s\in\left[ \tau,\tau + 1\right] }\left\vert W(s) - W(\tau)  \right\vert  
&  \to 0 \quad\mbox{ as }\quad \tau\to \infty
\label{S3E7}\\
\sup_{s\in\left[  \tau,\tau+1\right]  }\left\vert \Phi(s)
-\Phi(\tau)  - W(\tau)  (s-\tau)\right\vert  
&  \to 0\quad\mbox{as}\quad\tau\to\infty\,, 
\label{S3E8}
\end{align}
On the other hand, integration of the last equation of (\ref{S3E1}) implies 
\[
\Psi(\tau + 1)  -\Psi(\tau)  =\int_{\tau}^{\tau+1}
\left[  \frac{1}{\left(  \Phi(s)  \right)^3}-1\right]ds\,. 
\]
which together with (\ref{S3E6}) gives that
\begin{equation}\label{S3E10}
\lim_{\tau\to\infty}\int_{\tau}^{\tau+1}\left[\frac{1}{\left(\Phi(s)\right)^{3}}-1\right] ds=0 \,. 
\end{equation}

Therefore $\left(\inf_{s\in\left[\tau,\tau+1\right]}\Phi(s)\right)\leq 2$ for $\tau$ sufficiently large, 
since otherwise there would be a contradiction. It then follows that $\lim_{\tau\to\infty}W(\tau)=0$. 
Indeed, otherwise (\ref{S3E7}), (\ref{S3E8}) would 
imply that $\Phi$ would become negative for large values of $\tau$. 
It then follows from (\ref{S3E8}) that:
\[
\sup_{s\in\left[  \tau,\tau+1\right]  }\left\vert \Phi(s)
-\Phi(\tau)  \right\vert \rightarrow0\quad \mbox{as}\quad \tau\to\infty
\]
and (\ref{S3E10}) then yields $\lim_{\tau\to\infty}\Phi(\tau)=1$ against the hypothesis of the lemma. 
The contradiction yields $E(\infty)=\infty$.

We now prove (\ref{C1E3}). Suppose that
\begin{equation}\label{C3E5}
\lim\sup_{\tau\to\infty}\left(\frac{1}{2\Phi^2}+\Phi\right)<\infty\,. 
\end{equation}
Since $E(\infty)=\infty$, it follows from the definition of $E$ 
in (\ref{S3E3}) that 
\begin{equation}\label{C3E4}
\lim_{\tau\to\infty}\left[  \Psi(\tau)\,  W(\tau)  \right]  =\infty\,. 
\end{equation}
Suppose that $\lim_{j\to\infty}\left\vert \Psi(  \tau_j)\right\vert =\infty$ 
for some sequence $\left\{  \tau_j\right\}$ with 
$\lim_{j\to\infty}\tau_j=\infty$. Due to (\ref{C3E5}), the right-hand side of the last equation in (\ref{S3E1}) 
is bounded. Therefore, 
there exists a subsequence of $\left\{  \tau_j\right\}$, that we label in 
the same manner for simplicity, such that:
\[
\inf_{\tau\in\left[\tau_j-1,\tau_j + 1 \right]} \left\vert \Psi(\tau) \right\vert \to\infty\,.
\]
This means that $\inf_{\tau\in\left[  \tau_j-1,\tau_j + 1\right]}
\left\vert \frac{d^{2}\Phi}{d\tau^{2}}\right\vert \to\infty$. 
Therefore $\sup_{\tau\in\left[  \tau_j-1,\tau_j+1\right]}\Phi(\tau)\to\infty$. 
Indeed, if $\Phi(  \tau_j\pm 1)$ are bounded, and since we can determine uniquely $\Phi(\tau)$ 
in the interval $\tau\in(  \tau_j - 1,\tau_j + 1)$ using 
$\Phi(\tau_j\pm 1)$ and $\frac{d^2\Phi}{d\tau^2}(\tau)$, $\tau\in [\tau_j-1,\tau_j+1]$ the claim 
follows. This contradicts (\ref{C3E5}), whence $\lim\sup_{\tau\to\infty}\left\vert \Psi(\tau)\right\vert <\infty$. 
Then, (\ref{C3E4}) implies $\lim_{\tau\to\infty}W(\tau)=\infty$.
 Using the first equation in (\ref{S3E1}) it then follows that 
$\Phi(\tau)$ is unbounded, and this contradicts again 
(\ref{C3E5}), whence (\ref{C1E3}) follows.

It only remains to prove (\ref{C1E3a}). We first notice that
\begin{equation}\label{C3E7}
\lim\inf_{\tau\to\infty}\Phi(\tau) \leq 2\,. 
\end{equation}
Indeed, otherwise it follows from the last two equations of (\ref{S3E1}) that 
$\frac{d^3\Phi}{d\tau^3}\leq-\frac{1}{2}$, this implies that $\Phi(\tau)$ vanishes for a finite value of $\tau$. 
This contradicts Lemma~\ref{global:existence1}.

We now claim that $\lim\sup_{\tau\to\infty}\Phi\geq1$. We argue by
contradiction. Suppose that
\begin{equation}\label{C3E6}
\lim\sup_{\tau\to\infty}\Phi<1\,,
\end{equation}
then, using the last equation in (\ref{S3E1}), we have $\frac{d\Psi}{d\tau}>\delta>0$ 
for $\tau$ large enough. This means that $\lim_{\tau\to\infty}\Psi(\tau)  =\infty$. 
Using again (\ref{S3E1}), we obtain $\lim_{\tau\to\infty} W(\tau)  =\infty$ and  
$\lim_{\tau\to\infty}\Phi(\tau)  =\infty$, this contradicts (\ref{C3E6}) and therefore implies 
$\lim\sup_{\tau\to\infty}\Phi\geq 1$.

We now claim that $\lim\sup_{\tau\to\infty}\Phi>\lim\inf_{\tau\to\infty}\Phi(\tau)$. 
Suppose to the contrary that
$\lim_{\tau\to\infty}\Phi(\tau)=\ell_{\infty}\geq 1$ is defined. If $\ell_{\infty}>1$, 
we obtain a contradiction because (\ref{S3E1}), in that case, implies that $\Phi(\tau)$ vanishes at some 
finite $\tau$, but this contradicts Lemma~\ref{global:existence1}. And the case $\ell_{\infty}=1$ is against the assumption of the lemma.

We have then $\lim\sup_{\tau\to\infty}\Phi(\tau)>\lim\inf_{\tau\to\infty}\Phi(\tau)$. 
This implies the existence of sequences of minima and maxima of 
$\Phi$, namely, $\{  \tau_n^+ \}$, $\{ \tau_n^-\}$ and $\{ \varepsilon_n \}$, 
such that $\lim_{n\to\infty}\tau_n^+=\lim_{n\to\infty}\tau_n^-=\infty$, 
$\varepsilon_n>0$ for all $n$, and such that $W(\tau_n^+)=W(  \tau_n^-)=0$ with 
$\max_{\tau\in\left[ \tau_n^+ - \varepsilon_n,\tau_n^+ +\varepsilon_n\right]}\Phi(\tau)=\Phi(\tau_n^+)$, 
$\min_{\tau\in\left[  \tau_n^+-\varepsilon_n,\tau_n^+ + \varepsilon_n\right]  } \Phi(\tau)=\Phi(\tau_n^-)$ and
 $\lim_{n\to\infty}\Phi(\tau_n^+) =\lim\sup_{\tau\to\infty}\Phi(\tau)$, 
$\lim_{n\to\infty}\Phi(\tau_n^-) =\lim\inf_{\tau\to\infty}\Phi(\tau)$. 
Suppose that $\lim\inf_{\tau\to\infty}\Phi(\tau) >0$. 
Due to (\ref{C3E7}) and the definition of $E$ in (\ref{S3E3}) it then follows
that
\[
E\left(  \tau_{n}^{-}\right)  =\frac{1}{2(\Phi(  \tau_n^-))^2}+\Phi(\tau_n^-)  
\leq\frac{1}{\left[  \lim\inf_{\tau\to\infty}\Phi(\tau)  \right]^{2}}+3
\]
for $n$ sufficiently large. This contradicts that $E(\infty)=\infty$ and the second 
formula in (\ref{C1E3a}) follows. Suppose now that $\lim\sup_{\tau\to\infty}\Phi(\tau)<\infty$. 
Using again the definition of $E$ and that $\lim\sup_{\tau\rightarrow\infty}\Phi\geq1$ we 
obtain
\[
E(\tau_n^+)  =\frac{1}{2(\Phi(\tau_n^+))^2}
+\Phi(\tau_n^+)  \leq 1 + 2\lim\sup_{\tau\to\infty}\Phi(\tau)
\]
which also contradicts that $E(\infty)=\infty$.
\end{proof}

\subsection{Oscillations: The formal description}\label{section:3.2}
We now formally describe in some detail the oscillatory solutions of (\ref{WilsonJones}).  
Such solutions have increasing amplitude of oscillation as $\tau$ 
increases (see Lemma~\ref{bh:unstable:manifold}). When this amplitude becomes rather 
large as $\tau\to\infty$, one expects the behaviour of the solution to be 
dominated by (\ref{SEphi-inf}). This equation can be solved explicitly:
\begin{equation}\label{polynomial}
\Phi(\tau)=-\tau^{3}/6+c_{1}\tau^{2}+c_{2}\tau\,. 
\end{equation}
Thus the solution in such regimes resembles a third order polynomial near a local maximum. 
For larger values of $\tau$ the magnitude of the solution 
becomes small and its behaviour is dominated instead by (\ref{SEphi0}), 
but, as it is proved in Lemma~\ref{global:existence1}, the solution does not 
vanish for finite values of $\tau$ and it then increases back again to a 
larger amplitude. How this happens is described by (\ref{SEphi0}). 
Thus the function $\Phi(\tau)$ can be described by
 means of alternating regimes, where either $\Phi(\tau)$ becomes very 
large or $\Phi(\tau)$ is close to zero. As it will be seen later, the matching between 
such regimes will require that, when $\Phi$ becomes small for increasing $\tau$, the polynomial
of the form (\ref{polynomial}) that is asymptotically close to the solution, has a simple zero. 
On the other hand, when the solution leaves a local minimum (sufficiently close to $0$ 
for $\tau$ large enough) the polynomial (\ref{polynomial}) that it approaches asymptotically has a double zero. 
These, linear and quadratic, behaviours near $\Phi\sim 0$ result from the analysis of (\ref{SEphi0}),
this is done in Appendix~2, see Theorem~\ref{matchsolution}.

Since $\Phi(\tau)$ is oscillatory (see (\ref{C1E3a})), there exists an increasing sequence 
$\{\tau_n^+\}$ of local maxima, thus having $\frac{d\Phi}{d\tau}(\tau_n^+)=0$. For simplicity of notation,
for each $n$, we define
\[
L_n=\Phi(\tau_n^+)  \,.
\]
Then Lemma~\ref{bh:unstable:manifold} (\ref{C1E3a}) implies $\lim_{n\to\infty}L_n=\infty$. 
As indicated above, $\Phi(\tau)$ must have a double zero, at least to the leading order as $L_n\to\infty$. 
Using the fact that $\Phi$ solves (\ref{SEphi-inf}) to the leading order, it 
can be seen by means of an algebraic computation, that the only possibility is 
to have the following asymptotics if $L_n\to\infty$:
\begin{equation}\label{P1E1}
\Phi(\tau_n^+)=L_n\,,\ \frac{d}{d\tau}\Phi(\tau_n^+)=0\,,\ \frac{d^{2}}{d\tau^2}\Phi(\tau_n^+)=-a L_n^\frac{1}{3}\,,
\end{equation}
since $\Phi(\tau)$ must have a double zero at some $\tau<\tau_n^+$, to leading order as $L_n\to\infty$, 
$a$ will be chosen later in order to fulfil this condition. 
We introduce the scaling, suggested by (\ref{P1E1}),
\begin{equation}\label{2order1:scale}
\Phi(\tau)  =L_n \bar{\Phi}_n(\eta_n)\,,\quad \eta_n=\frac{\tau-\tau_n^+}{L_n^\frac{1}{3}}
\end{equation}
where $\bar{\Phi}_n(\eta_n)$ satisfies
\begin{equation}\label{W1E1}
\frac{d^{3}}{d\eta_n^3}\bar{\Phi}_n + 1=\frac{1}{L_n^3}\frac{1}{\left(\bar{\Phi}_n\right)^{3}}\,,  
\end{equation}
that, to leading order, becomes
\begin{equation}\label{lim-eq:osci:outer}
\frac{d^{3}}{d\eta_n^3}\bar{\Phi}_n + 1 = 0\,,
\end{equation}
with initial conditions
\begin{equation}\label{lim-eq:osci:outer:init}
\bar{\Phi}_{n}(0)=1\,,\ \frac{d\bar{\Phi}_n}{d\eta_n}(0)=0\,,\ \frac{d^2\bar{\Phi}_n}{d\eta_n^2}(0)=-a\,.
\end{equation}
Solving (\ref{lim-eq:osci:outer})-(\ref{lim-eq:osci:outer:init}) gives the 
approximation $\bar{\Phi}_n(\eta_n)=P_{\infty}(\eta_n)$, where
\begin{equation}\label{polynomial:eta}
P_{\infty}(\eta)=-\frac{1}{6}\eta^{3}-\frac{a}{2}\eta^{2}+1\,.
\end{equation}
Imposing now that $P_{\infty}$ has a double zero at 
a negative $\eta_0$ requires:
\[
\bar{\Phi}_{n}(\eta_{0})=-\frac{1}{6}\eta_{0}^{3}-\frac{a}{2}\eta_0^2 + 1=0\,,
\quad \bar{\Phi}_n^{\prime}(\eta_0)=-\frac{1}{2}\eta_0^2-a\eta_0=0\,,
\]
i.e.
\begin{equation}\label{a:value}
\eta_{0}=-2a\,,\quad  a=\left(  \frac{3}{2}\right)^{\frac{1}{3}}\,.
\end{equation}
On the other hand, the polynomial $P_{\infty}(\eta)$ vanishes at a 
positive value of $\eta$, namely at $\eta_{a}:=\left(\frac{3}{2}\right)^{\frac{1}{3}}$. 
We can then approximate the function $\Phi(\tau)$ in the 
intervals where $\tau \in \left(  \tau_n^+ + \eta_0 L_n^\frac{1}{3}, \tau_n^+ + \eta_a L_n^\frac{1}{3}\right)$ 
by means of (\ref{2order1:scale}) and (\ref{polynomial:eta}) with (\ref{a:value}).
The approximations of $\bar{\Phi}$ by (\ref{polynomial:eta}) near $\eta_a^-$ and near $\eta_0^+$ are then as follows 
\begin{equation}\label{linear:bh}
\bar{\Phi}_n(\eta_n)\sim - K(\eta_n-\eta_a)\quad\mbox{as}\quad\eta_n \to \eta_a^-
\quad\mbox{with}\quad 
K=-\frac{d}{d\eta_n}\bar{\Phi}_n(\eta_a)=\left(  \frac{3}{2}\right)^{\frac{5}{3}}>0
\end{equation}
and 
\[
\bar{\Phi}_n(\eta_n)\sim \frac{a}{2}(\eta_n-\eta_0)^2
\quad\mbox{as}\quad \eta_{n}\to\eta_0^+,  
\quad\mbox{since}\quad \frac{d^2}{d\eta_n^2}\bar{\Phi}_n(\eta_0)=a=\left(  \frac{3}{2}\right)^{\frac{1}{3}}\,. 
\]

In order to describe the function $\Phi(\tau)$ for the values of $\tau$ where 
$\Phi$ becomes small we introduce an inner layer variable near $\eta_a$ for 
every $n$. Using (\ref{linear:bh}) we can infer that this inner layer is characterised by 
\begin{equation}\label{innerlayer}
\eta_n=\eta_a + \frac{1}{L_n^3}\zeta_n\quad\bar{\Phi}(\eta_n)=\frac{1}{L_n^3}\varphi_n(\zeta_n) 
\end{equation}
with $\varphi(\zeta_n)$ satisfying, to leading order if $n$ is large enough, 
the equation
\begin{equation}\label{lim-eq:osci:inner}
\frac{d^3\varphi_n}{d\zeta_n^3}=\frac{1}{\varphi_n^3}
\end{equation}
with the matching condition
\begin{equation}\label{lim-eq:osci:inner:init}
\varphi_n (\zeta_n)  \sim -K\zeta_n + \frac{\log(L_n^3)}{2K^3}+ \frac{1}{4K^3} + o(1)  
\quad \mbox{as}\quad \zeta_n\to -\infty\,,
\end{equation}
where the logarithmic correction and the next order has been obtained, 
computing the next order in the asymptotics of the solution of (\ref{W1E1}) and (\ref{lim-eq:osci:outer:init}). 
More precisely, this correction is given by using the leading behaviour (\ref{linear:bh}), namely  
\[
\int_0^{\eta_n}\int_0^{s_1}\int_0^{s_2}\frac{1}{\left(\bar{\Phi}_n(s_3)\right)^3}\, ds_3\,ds_2\, ds_1
=-\frac{1}{2K^3}\log\left(\left| \eta_n - \eta_a\right| \right)  + 
\frac{3}{2K^3}  + o(1)\quad \mbox{as}\quad\eta_n\to\eta_a^-
\]
where $\bar{\Phi}_n(s_3)$ is as in (\ref{polynomial:eta}).

Using Theorem~\ref{matchsolution} in Appendix~2, it follows that there exists a 
unique solution of (\ref{lim-eq:osci:inner}), (\ref{lim-eq:osci:inner:init}) 
given by $\varphi_n(\zeta_n)  =\varphi\left(\zeta_n -\frac{\log(L_n^3)}{2K^4}- \frac{3}{2K^4} \right)$. 
Therefore:
\begin{equation}\label{matchforward}
\varphi_n (\zeta_n)  \sim \Gamma\zeta_n^2
\quad \mbox{as} \quad \zeta_n \to\infty\,,
\end{equation}
where $\Gamma>0$ would in fact be fixed by the matching region.  
We observe that the matching condition (\ref{lim-eq:osci:outer:init}) is valid for 
$\left\vert \zeta_n \right\vert \gg \log(L_n)$, $\left\vert\zeta_n\right\vert \ll L_n^3$, $\zeta_n<0$. 
On the other hand, the matching condition (\ref{matchforward}) is valid for $\left\vert \zeta_n\right\vert \gg\log(L_n)$, 
$\zeta_n>0$. In order to determine the maximal region of validity for this asymptotics we need to find 
the size of the region where the second term on the right-hand side of 
(\ref{W1E1}) becomes relevant. A standard dominated balance argument indicates 
that this happens for $\eta_{n}-\eta_{a}=O(L_{n}^{3})$. 
We then use the following change of variables:
\begin{equation}\label{W1E4}
\bar{\Phi}_n(\eta_n) = \left(\frac{2\Gamma}{a}\right)^3 L_n^9 \bar{\Phi}_{n+1}(\eta_{n+1})\,,
\quad \eta_{n+1} + 2a = \frac{a}{2\Gamma}\frac{\eta_n-\eta_a}{L_n^3} 
\end{equation}
where $a$ is as in (\ref{a:value}). The particular rescaling has been chosen 
in order to have $\bar{\Phi}_{n+1}(0)=1$ to the leading order. Notice that, 
using (\ref{W1E1}) we obtain:
\begin{equation}\label{W1E2}
\frac{d^3}{d\eta_{n+1}^3}\bar{\Phi}_{n+1}+1=\left(  \frac{a}{2\Gamma}\right)^{9}\frac{1}{L_n^{30}}\frac{1}{\left(\bar{\Phi}_{n+1}\right)^3}
\end{equation}
Using (\ref{innerlayer}) and (\ref{matchforward}) we obtain the matching 
condition for $\bar{\Phi}_{n+1}(\eta_{n+1})$:
\begin{equation}\label{W1E3}
\bar{\Phi}_{n+1}(\eta_{n+1})\sim\frac{a}{2}\left(  \eta_{n+1}+2a\right)^2
\quad \mbox{as}\quad \eta_{n+1}\to -2a
\end{equation}
and the solution of (\ref{W1E2}), (\ref{W1E3}) to the leading order gives the 
approximation $\bar{\Phi}_{n+1}(\eta_{n+1})=P_{\infty}(\eta_{n+1})$. 
Notice that we obtain $\bar{\Phi}_{n+1}(0)=1$ to the leading 
order, as expected. Note also that (\ref{2order1:scale}), (\ref{W1E4}) imply:
\[
\Phi(\tau)  =\left(  \frac{2\Gamma}{a}\right)^3 L_n^{10}\bar{\Phi}_{n+1}(\eta_{n+1})\,,
\quad \eta_{n+1}=-2a+\frac{a}{2\Gamma}\frac{(\tau-\tau_n^+) -\eta_a L_n^\frac{1}{3}}{L_n^{\frac{10}{3}}}\,.
\]
Equation (\ref{W1E2}) is equivalent to (\ref{W1E1}) and therefore we have
\begin{equation}\label{W1E5}
L_{n+1}=\left(  \frac{2\Gamma}{a}\right)^3 L_n^{10} 
\end{equation}
On the other hand, since $\eta_{n+1}=0$ at $\tau=\tau_{n+1}^+$ we obtain:
\begin{equation}\label{W1E7}
\tau_{n+1}^+=\tau_n^+ + \eta_a L_n^\frac{1}{3} + 4\Gamma L_n^{\frac{10}{3}}
\end{equation}

The sequences (\ref{W1E5}) and (\ref{W1E7}) describe how the increase in the amplitude of $\Phi$ takes place. 
Notice that:
\begin{equation}\label{W1E8}
L_n=\frac{a^{\frac{1}{3}}\exp\left(  C (10)^{n}\right)}{ (2\Gamma)^{\frac{1}{3}}} (1+o(1))\,,
\quad \tau_n^+=2a^{\frac{10}{9}} (2\Gamma)^{\frac{2}{3}}\exp\left(\frac{C}{3} (10)^{n}\right) (1+o(1))  
\quad \mbox{as} \quad n\to\infty
\end{equation}

This formula can be easily obtained from (\ref{W1E5}) using the fact that 
$Z_n=\log(L_n)$ solves the linear recursive equation
\[
Z_{n+1}=10 Z_n + 3\log\left(  \frac{2\Gamma}{a}\right)\,.
\]
Moreover, suppose that we denote by $M_n$ the minimum values of $\Phi(\tau)$ 
in the interval $(\tau_n^+,\tau_{n+1}^+)$, and that these minima are reached 
at the points $\tau_n^-\in(\tau_n^+,\tau_{n+1}^+)$. Then, using (\ref{2order1:scale}) and
(\ref{innerlayer}), as well as the arguments yielding (\ref{W1E7}) we obtain:
\begin{equation}\label{W1E8a}
M_n=\frac{\min_{\mathbb{R}}\varphi(\zeta)}{L_n^2}(1+o(1)) \,,
\quad \tau_n^-=\left[  2a^{\frac{10}{9}} (2\Gamma)^{\frac{2}{3}} + 
\frac{\eta_a a^{\frac{1}{9}}}{(2\Gamma)^{\frac{1}{9}}}\right]  
\exp\left(  \frac{C}{3} (10)^n \right) ( 1+o(1))\quad \mbox{as} \quad n\to\infty
\end{equation}
where $\varphi$ is as in Theorem~\ref{matchsolution}.

\subsection{Oscillations: The rigorous construction}\label{section:3.3}
In this Section we prove that every solution of (\ref{WilsonJones}) 
such that $\Phi(\tau)$ does not converge to $1$ as $\tau\to\infty$ oscillates with increasing amplitude 
as $\tau\to\infty$ in the form described above by the formal asymptotics. Namely, we prove the following theorem:

\begin{theorem}\label{mainoscillation}
Suppose that $\Phi(\tau)$ is a solution of (\ref{WilsonJones}) that is not included in the stable manifold of the point
$P_{s}$. 
There exists two increasing sequences $\{\tau_n^-\}$, $\{\tau_n^+\}$ satisfying 
$\tau_n^+<\tau_n^-<\tau_{n+1}^+$ and $\lim_{n\to\infty}\tau_n^+=\infty$, such that:
\[
\frac{d\Phi}{d\tau}(\tau_n^-)  =\frac{d\Phi}{d\tau}(\tau_n^+)=0\quad \,,
\quad\frac{d\Phi}{d\tau}(\tau)  >0 \quad \mbox{in} \quad (\tau_n^-,\tau_n^+)  \,,
\quad \Phi_{\tau}\left(\tau\right)  <0 \quad\mbox{in }\quad(\tau_n^+,\tau_{n+1}^- )
\]
Moreover, for every $n$, let $L_n$ and $M_n$ be
\begin{equation} \label{defLnMn}
L_n:=\Phi(\tau_n^+)  \,,\quad M_n:=\Phi(\tau_n^-) 
\end{equation}
Then, the asymptotics (\ref{W1E8}) and (\ref{W1E8a}) hold, and also:
\begin{equation}\label{secder}
\frac{d^2\Phi(\tau_n^+)}{d\tau^2}  =-\left(  \frac{3}{2}\right)^{\frac{1}{3}}L_n^\frac{1}{3}(1+o(1))
\quad \mbox{as} \quad n\to\infty
\end{equation}
\end{theorem}

\begin{remark}
Notice that Theorem~\ref{mainoscillation} implies that every solution of 
(\ref{WilsonJones}) that does not approach $\Phi=1$ as 
$\tau\to\infty$ oscillates for large values of $\tau$ in the precise 
manner indicated in the Theorem.
\end{remark}

The proof of Theorem \ref{mainoscillation} will be decomposed in a series of Lemmas. 
We first observe that the existence of the sequences $\{  \tau_n^+\}$ and $\{  \tau_n^-\}$ 
is a direct consequence of Lemma~\ref{bh:unstable:manifold}.  Also, if $L _n$ and $M_n$ are as in 
(\ref{defLnMn}) for all $n$, then 
Lemma~\ref{bh:unstable:manifold} implies that $\lim_{n \to\infty}L_n=\infty$ and $\lim_{n\to\infty}M_n=0$. 
We define the sequence $\{a_n\}\subset \mathbb{R}$ by means of
\[
\frac{d^2}{d\tau^2}\Phi(\tau_n^+)= - a_n L_n^\frac{1}{3}\,,\quad\mbox{where}\quad a_n \geq 0\,
\]
and we shall then prove that $a_n\to a$, with $a$ as in (\ref{a:value}), as $n\to\infty$. 
This then implies (\ref{secder}).

We define a sequence of functions $\bar{\Phi}_n( \eta_n)$ as in (\ref{2order1:scale}), but  
where $\bar{\Phi}_{n}$ solves (\ref{W1E1}) with the initial conditions: 
\begin{equation}\label{Incond}
\bar{\Phi}_n(0)=1\,,\ \frac{d\bar{\Phi}_n}{d\eta_n}(0)=0\,,\ \frac{d^2\bar{\Phi}_n}{d\eta_n^2}(0)= - a_n\,.
\end{equation}
Observe that the only difference with (\ref{lim-eq:osci:outer:init}) is that the last condition depends on $n$. We define the approximating 
polynomials and their roots accordingly; letting $P_{a_n}(\eta)=-\frac{1}{6}\eta^{3}-\frac{a_n}{2}\eta^2 + 1$, 
we define $\eta_{a_n}$ as the solution of
\[
P_{a_n}(\eta_{a_n})=0\,,\quad \eta_{a_n}>0\,.
\]
We observe that when $a_n$ is very large then $\eta_{a_n}\sim 1/\sqrt{a_n}$, and 
otherwise $\eta_{a_n}$ is of order one. For that reason and in order to avoid studying several possible regimes for $a_n$, 
the scalings with $L_n$ would include a factor depending on $a_n$, as we shall see below.

We first derive some approximation formulae for the functions $\bar{\Phi}_n (\eta_n)$ in the region where they are not too small.
\begin{lemma}\label{Lemma1}
Let $\Phi$ a solution of (\ref{WilsonJones}) satisfying the 
assumption of Theorem \ref{mainoscillation}. Let $\bar{\Phi}_n (\eta_n)$ be as in (\ref{2order1:scale}). 
Then, there exists a $N>0$ independent of $n$ and of $a_{n}$ such that the
 estimates 
\begin{equation}\label{W1E9}
\bar{\Phi}_n(\eta_n)\geq\frac{1}{2\eta_{a_n}}\left(  \eta_{a_n}-\eta_n\right)\,, 
\end{equation}
\begin{equation}\label{W2E1}
\left\vert \bar{\Phi}_n(\eta_n) - \left(1-\frac{a_n\eta_n^2}{2}-\frac{\eta_n^3}{6}\right)  \right\vert 
\leq \frac{4(\eta_{a_n})^3}{L_n^3} \left\vert \log\left(  1-\frac{\eta_n}{\eta_{a_n}}\right)  \right\vert \,,
\end{equation}
\begin{align}
\left\vert \frac{d\bar{\Phi}_n (\eta_n) }{ d\eta_n} + a_n\eta_n + \frac{\eta_n^2}{2} \right\vert  
&  \leq \frac{4 (\eta_{a_n})^{3} }{ L_n^3} \frac{1}{ (\eta_{a_n} - \eta_n)  }\,,\label{W2E2}\\
\left\vert \frac{d^2\bar{\Phi}_n (\eta_n)}{d\eta_n^2} + a_n + \eta_n \right\vert  
&  \leq\frac{4 (\eta_{a_n} )^{3}}{L_n^3} \frac{1}{ (\eta_{a_n}-\eta_n)^{2}} \,,\label{W2E3}
\end{align}
hold as long as:
\begin{equation}\label{W2E3a}
\frac{(\eta_{a_n}-\eta_n)}{\left\vert \log\left(1-\frac{\eta_n}{\eta_{a_n}}\right)  \right\vert }
\geq N \frac{\eta_{a_n}^4}{L_n^3}\,.
\end{equation}
\end{lemma}

\begin{proof}
Direct integration of (\ref{W1E1}) with (\ref{Incond}) gives:
\begin{align}
\bar{\Phi}_n(\eta_n)  &  =1 - \frac{a_n\eta_n^2}{2} -\frac{\eta_n^3}{6} 
+ \frac{1}{L_n^3}\int_0^{\eta_n} \int_{0}^{s_{1}}\int_0^{s_2}\frac{1}{(\bar{\Phi}_n(s_3))^3}ds_3\,ds_2\, ds_1\,,\nonumber\\
\frac{d\bar{\Phi}_n(\eta_n)}{d\eta_n}  &  = - a_n\eta_n - \frac{\eta_n^2}{2} +
 \frac{1}{L_n^3}\int_0^{\eta_n}\int_0^{s_1} \frac{1}{ (\bar{\Phi}_n(s_2))^3}ds_2\,ds_1\,,\label{int} \\
\frac{d^2\bar{\Phi}_n(\eta_n)}{d\eta_n^2}  &  =
 -a_n-\eta_n +\frac{1}{L_n^3} \int_0^{\eta_n}\frac{ds_1}{(\bar{\Phi}_n (s_1))^3}\,. \nonumber
\end{align}
We now claim that (\ref{W1E9})-(\ref{W2E3}) hold if (\ref{W2E3a}) is satisfied for some $N>0$ independent of $L_n$ and of $a_n$. 
Indeed, this is proved by means of a continuation argument. The inequality (\ref{W1E9}) is satisfied for $\eta_n=0$. 
On the other hand, as long as this inequality is satisfied we have:
\[
\left\vert 
\bar{\Phi}_n (\eta_n)
-\left(  1-\frac{a_n \eta_n^2}{2}-\frac{\eta_n^3}{6}\right)  
\right\vert 
\leq \frac{8 (\eta_{a_n})^3}{L_n^3}\int_0^{\eta_n} \int_0^{s_1}
\int_0^{s_2}\frac{1}{ (\eta_{a_n}-s_3)^3} ds_3 \, ds_2\, ds_1 
\]
and hence (\ref{W2E1}) also follows. Due to the convexity of the polynomial $P_{a_n}(\eta)$ 
we have $P_{a_n}(\eta_n)\geq\frac{(\eta_{a_n}-\eta_n) }{\eta_{a_n}}$. It then follows that: 
\[
\bar{\Phi}_n(\eta_n)\geq\frac{(\eta_{a_n}-\eta_n) }{\eta_{a_n}}-\frac{4( \eta_{a_n})^3}{L_n^3} 
\left| \log\left(  1-\frac{\eta_n}{\eta_{a_n}}\right)  \right|
\]
which implies (\ref{W1E9}). Moreover, using (\ref{int}), we obtain that (\ref{W2E1}), (\ref{W2E2}) and 
(\ref{W2E3}) hold as long as (\ref{W2E3a}) holds.
\end{proof}

It is now convenient to reformulate this result using a new set of variables, 
which is better suited for the study of the boundary layer where the term
$\frac{1}{L_n^3}\frac{1}{(\bar{\Phi}_n)^3}$ in (\ref{W1E1}) becomes relevant. Namely
we take:
\begin{equation}\label{defphi}
\bar{\Phi}_n (\eta_n) = \frac{(\eta_{a_n})^3}{L_n^3}\varphi_n (\zeta_n) \,,
\quad \eta_n-\eta_{a_n}=\frac{(\eta_{a_n})^{4}}{L_n^3}\zeta_n
\end{equation}
and (\ref{W1E1}) becomes
\begin{equation}\label{phin}
\frac{d^3\varphi_n}{d\zeta_n^3} + \frac{(\eta_{a_n})^9}{L_n^6}=\frac{1}{(\varphi_n)^3}
\end{equation}

The next step is to change into the variables which transform (\ref{SEphi0}) into a planar system of ODEs. 
We use the transformation (\ref{SEphi0-trans}) of Appendix~2, that now reads
\begin{equation}\label{W2E9}
\frac{d\varphi_n}{d\zeta_n}=\varphi_n^{-\frac{1}{3}}u_n\,,
\quad \frac{d^2\varphi_n}{d\zeta_n^2}=\varphi_n^{-\frac{5}{3}}v_n\,,
\quad z_n=\int_{-R_{M,n}}^{\zeta_n}\frac{ds}{(\varphi_n (s))^{\frac{4}{3}}}\,, 
\end{equation}
where
\begin{equation}\label{W2E8}
R_{M,n}=M\left\vert \log\left(\frac{\eta_{a_n}}{L_n}\right)  \right\vert\quad\mbox{for some}\quad M>1\,,
\end{equation}
(in this way, $\zeta_n$ very negative is in the matching region, see (\ref{W2E3a})). This transforms
(\ref{phin}) into the system
\begin{equation}\label{ODEn3}
\frac{d\varphi_n}{dz_n}=u_n\varphi_{n} \,,\ \frac{du_n}{dz_n}=v_n+\frac{1}{3}u_n^2\,,
\ \frac{dv_n}{dz_n}=1+\frac{5}{3}u_n\,, \ v_n-\frac{(\eta_{a_n})^{9}}{L_n^6}(\varphi_n)^3\,,
\end{equation}
that can be considered as a perturbation of the system (\ref{phi0system}) for $n$ large enough. 
We will use $u=u_n$ as independent variable and one should understand in the following, with some
abuse of notation, that $\varphi_n=\varphi_n (u)$, $v_n=v_n (u)$, $z_n=z_n (u)$, then
\begin{equation}\label{W3E2}
\frac{d\varphi_n}{du}=\frac{u\, \varphi_n}{v_n + \frac{1}{3} u^2}
\,,\quad  \frac{dv_n}{du}=\frac{1+\frac{5}{3}u\,v_n -\frac{(\eta_{a_n})^9}{L_n^6}(\varphi_n)^3}{v_n + \frac{1}{3} u^2}\,, 
\quad \frac{dz_n}{du}=\frac{1}{v_n + \frac{1}{3} u^2}\,.
\end{equation}
The following lemma is a translation of Lemma~\ref{Lemma1} into this system in the matching region.
\begin{lemma}\label{Lemma2}
\begin{enumerate}
\item Suppose that $(v_n + \frac{1}{3} u_n^2)$ is large. Then we can define $\varphi_n(u)$, 
$v_n(u)$, $z_n(u) $ by means of (\ref{W2E9}) with $u=u_n$. 

\item For any $\varepsilon>0$ arbitrarily small, there exist
$M_0=M_0(\varepsilon)$ and $n_0=n_0(\varepsilon,M)$ such that for $M\geq M_0$, $n\geq n_0$ we have:
\begin{align}
\left| \varphi_n (u^{\ast})  - K_n R_{M,n}\right|  &\leq \varepsilon R_{M,n} \,, 
\quad\left| v_n (u^{\ast})\right| \leq \varepsilon\,,\quad z_n=0 \label{W3E2a}\\
\mbox{for some} \quad u^{\ast} & \in\left[  -K_n^{\frac{4}{3}}(1+\varepsilon) (R_{M,n})^{\frac{1}{3}}, 
-K_n^{\frac{4}{3}}(1-\varepsilon)  (R_{M,n})^{\frac{1}{3}}\right] \nonumber
\end{align}
with
\[
K_n=\left(  a_n\eta_{a_n}^2+\frac{\eta_{a_n}^3}{2}\right)\,.
\]
\end{enumerate}
\end{lemma}
\begin{proof}
The possibility of defining the functions $\varphi_n (u) $, $v_n (u)$, $z_n (u)$ is just a consequence of 
the second equation in (\ref{ODEn3}) and the Implicit Function Theorem. Using Lemma~\ref{Lemma1} we obtain:
\[
\left\vert \varphi_n + K_n \zeta_n \right\vert 
\leq\varepsilon \left\vert\zeta_n\right\vert \,,\quad \left\vert \frac{d\varphi_n}{d\zeta_n}+ K_n \right\vert 
\leq C (a_n \eta_{a_n}^2 )  \left(\frac{\eta_{a_n}}{L_n}\right)^3 \left\vert \zeta_n\right\vert
+\frac{4}{\left\vert \zeta_n \right\vert }
\]
and
\[
\left\vert \frac{d^2 \varphi_n}{d\zeta_n^2} + \left(  \frac{a_n\eta_{a_n}^5}{ L_n^3} + \frac{\eta_{a_n}^6}{L_n^3}\right)\right\vert 
\leq C \frac{\eta_{a_n}^9}{L_n^6}\left\vert \zeta_n \right\vert +\frac{4}{\left\vert \zeta_n\right\vert^2}
\]
as long as $n$ is sufficiently large and 
\[
C \frac{L_n^3}{\eta_{a_n}^4}>\left\vert \zeta_n \right\vert \geq M\left\vert \log\left(  \frac{\eta_{a_n}}{L_n}\right)  \right\vert \,,
\]
where $C>0$ is independent of $n$ (the inequality on the left-hand side is used in the first and second inequalities). 
Moreover, using that the sequence $a_n\eta_{a_n}^{2}$ is bounded, as well as the fact that $\eta_{a_n}/L_n\to 0$ as $n\to\infty$, we obtain that at $\zeta_n = -R_{M,n}$ 
\[
\left\vert \varphi_n - K_n R_{M,n}\right\vert \leq\varepsilon R_{M,n}\,, 
\quad \left\vert \frac{d\varphi_n}{d\zeta_n} + K_n \right\vert \leq \varepsilon  \quad \mbox{and}
\quad  \left\vert \frac{d^2\varphi_n}{d\zeta_n^2}\right\vert \leq \frac{C\eta_{a_n}^3}{ L_n^3} + \frac{4}{R_{M,n}^2}
\]
hold, where $\varepsilon$ can be made arbitrarily small if we choose $M$ large enough and $n\geq n_0(M)$ sufficiently large too.
Thus, the functions $(\varphi_n,u_n,v_n)$ satisfy
\begin{equation}\label{W3E1}
\left\vert \varphi_n - K_n R_{M,n}\right\vert \leq\varepsilon R_{M,n}\,,
\quad  \left\vert u_n + K_n^{\frac{4}{3}} (R_{M,n})^{\frac{1}{3}}\right\vert 
\leq\varepsilon (R_{M,n})^{\frac{1}{3}} \,, \quad \left\vert v_n \right\vert 
\leq\varepsilon \quad \mbox{at} \ \zeta_n=-R_{M,n}
\end{equation} 
where $\varepsilon$ can be made arbitrarily small if $M$ and $n$ are large enough. 
Defining $u^{\ast}$ as the value of $u_n$ at which $\zeta_n=-R_{M,n}$ we obtain (\ref{W3E2a}).
\end{proof}

We now need to prove a stability lemma for the system (\ref{W3E2}). The 
behaviour of $(\varphi_n, u_n, v_n)$ is expected to be 
similar to that of the separatrix parametrised by $\left\{(u,v): \ v=\bar{v}(u)  \right\}$ 
of Lemma~\ref{separa} of Appendix~2 and that gives the linear and parabolic behaviours in terms of $\varphi$, 
see Theorem~\ref{matchsolution}. Before we go into such analysis, we prove some results for the following auxiliary functions.

\begin{lemma}\label{GH}
Let us define functions $G$ and $H$ by means of
\begin{equation}\label{W3E9}
G(u)  = \exp \left(\int_0^u\frac{s}{\bar{v}(s)  + \frac{s^{2}}{3}} ds \right)  \,,
\quad  
H(u) = \exp \left( \int_{0}^{u}\frac{\left(-1 + \frac{5s^{3}}{9} \right)}{\left(\bar{v}(s) + \frac{s^{2}}{3} \right)^{2}}ds\right) \,.
\end{equation}
Then, the following asymptotics hold
\begin{align}
G(u)   &  = K_{+}(u)^{\frac{6}{5}}\left(1+O\left(  \frac{1}{u^{\varepsilon_{0}}}\right)  \right)  
\quad \mbox{as}\quad u\to\infty \quad  \mbox{and}\label{A1E1}\\
G(u)   &  = K_{-}\left\vert u\right\vert ^{3}\left(  1+O\left(\frac{1}{\left\vert u\right\vert ^{\varepsilon_{0}}}\right)  \right)
\quad \mbox{as} \quad u\to -\infty\,, \label{A1E2}
\end{align}
with $\varepsilon_{0}>0$ and where $K_+$, $K_->0$. We have also that $G(u)>0$ for all 
$u\in\mathbb{R}$ and we can write the following global estimate:
\begin{equation}\label{W4E3}
\frac{1}{C_{\ast}}\left(  1+\left\vert u\right\vert^{\beta(u)}\right)  
\leq G(u)  \leq 
C_{\ast}\left(  1+\left\vert u\right\vert^{\beta(u)  }\right) 
\,, \quad 
\beta(u)=\frac{6}{5}\quad \mbox{if}\ u>0 \,,\quad  \beta(u)=3\quad \mbox{if}\ u<0 
\end{equation}
where $C_{\ast}>1$. Moreover,
\begin{align*}
H(u)   &  =\bar{K}_{+}(u)^{\frac{4}{5}}\left(1+O\left(  \frac{1}{u^{\varepsilon_{0}}}\right)  \right) 
\quad \mbox{as}\quad u\to\infty \quad \mbox{and}   \\
H(u)   &  =\bar{K}_{-}\left\vert u\right\vert^{5}
\left(1+O\left(  \frac{1}{\left\vert u\right\vert^{\varepsilon_{0}}}\right) \right)  
\quad \mbox{as}\quad u\to -\infty \,,
\end{align*}
with $\varepsilon_{0}>0$ and where $\bar{K}_+$, $\bar{K}_->0$. And, as above, there exists a 
$\bar{C}_{\ast}>1$ such that
\begin{equation}\label{W4E4}
\frac{1}{\bar{C}_{\ast}}\left(  1 + \left\vert u\right\vert^{\gamma\left(u\right)  }\right)  
\leq H(u)  \leq\bar{C}_{\ast} \left(1+\left\vert u\right\vert^{\gamma(u)} \right)  
\,,\quad \gamma(u)  =\frac{4}{5}\quad \mbox{if}\quad u>0 \,,\quad \gamma(u)=5\quad \mbox{if}\quad u<0 \,.
\end{equation}
\end{lemma}
\begin{proof}
This result is just a consequence of the asymptotics of the function $\bar{v}=\bar{v}(u)$ 
obtained in the Appendix~2, see Lemma~\ref{separa}, (\ref{W3E6c}).
\end{proof}

\begin{lemma}\label{BL}
Suppose that $\varphi_n (u)$ and $v_n(u)$ are defined as in Lemma \ref{Lemma2}. Then, for any $\varepsilon>0$ there
 exists $M_0=M_0(\varepsilon)$ and $n_0=n_0(\varepsilon,M)$ such that, for $M\geq M_0$ and $n\geq n_0$ the following inequalities hold:
\begin{equation}\label{W3E7}
\left\vert (v_n -\bar{v})(u)  \right\vert
\leq \frac{C_1 \varepsilon  }{\left(  1 + \left\vert u^{\ast}\right\vert^{5}\right)  } H(u)\,,
\end{equation}
with $C_1>0$ large enough and independent of $n$, $\varepsilon$ and $M$, and
\begin{equation}\label{W3E8}
\left\vert \frac{\varphi_n (u)}{\varphi_n (u^{\ast}) } - \frac{G(u)}{G(u^{\ast})}\right\vert 
\leq\frac{C\, C_1 \varepsilon}{\left(  1+\left\vert u^{\ast}\right\vert ^{2}\right)  }\frac{G(u)}{G(u^{\ast})}
\end{equation}
for $u\in\left[ u^{\ast},u^{\ast\ast}\right]$ with $u^{\ast\ast}=\frac{1}{2}\frac{L_n^3}{\eta_{a_n}^5}$ and $u^{\ast}$ is as in Lemma \ref{Lemma2} and $C>0$ is 
independent of $\varepsilon$, $n$, $M$ and $C_1$.
\end{lemma}
\begin{proof}
Using (\ref{W3E2}) and (\ref{W3E3}) we obtain
\begin{equation}\label{W3E4a}
\varphi_n (u)  =\varphi_n (u^{\ast})  \exp\left(\int_{u^{\ast}}^u \frac{s}{v_n (s) + \frac{s^2}{3}} ds\right)
\end{equation}
and
\[
\frac{d (v_n - \bar{v})}{du}= \frac{\left( -1 + \frac{5 u^3}{9}\right) (v_n - \bar{v})}{\left(v_n + \frac{u^2}{3}\right) \left(\bar{v}+\frac{u^2}{3}\right)}
-\frac{(\eta_{a_n})^9}{L_n^6}\frac{(\varphi_n)^3}{\left(  v_n + \frac{u^2}{3} \right)}\,.
\]
Then, (\ref{W3E8}) follows from (\ref{W3E4a}) and Lemma~\ref{GH} (\ref{W4E4}) (using Taylor) as long as (\ref{W3E7}) is satisfied.

Let us now prove (\ref{W3E7}). First observe that we can write 
\begin{equation}\label{W3E5}
\frac{d (v_n -\bar{v})}{du} = 
\frac{\left( -1 + \frac{5 u^3}{9}\right) (v_n - \bar{v}) }{ \left( \bar{v}+\frac{u^2}{3}\right)^2 } + f(u) 
\end{equation}
with
\[
f(u)=     \frac{\left( -1 + \frac{5 u^3}{9}\right) (v_n - \bar{v}) }{ \left( \bar{v}+\frac{u^2}{3}\right) }
\left( \frac{1}{ v_n + \frac{u^2}{3} } -\frac{1}{ \bar{v} + \frac{u^2}{3} } \right)  
  - \frac{(\eta_{a_n})^9}{L_n^6} \frac{(\varphi_n)^3}{ \left(  v_n+\frac{u^{2}}{3}\right) }\,.
\]
Thus, integrating (\ref{W3E5}) we obtain
\begin{equation}\label{W3E6a}
| (v_n -\bar{v})(u)| \leq 
2\varepsilon 
\exp\left(  \int_{u^{\ast}}^u 
\frac{ \left( -1 + \frac{5 s^3}{9}\right)}{\left(\bar{v}(s) + \frac{s^2}{3}\right)^2} ds\right)  +
\int_{u^{\ast}}^u 
\exp\left( \int_{\sigma}^u 
\frac{\left( -1 + \frac{5 s^3}{9} \right) }{\left(\bar{v}(s) + \frac{s^2}{3} \right)^2} ds \right)  
f(\sigma)  d\sigma\,.
\end{equation}
Also, using (\ref{W3E9}) and (\ref{W4E4}), we have
\begin{equation} \label{W3E6b}
\exp\left(  \int_{u^{\ast}}^u \frac{\left( -1 + \frac{5 s^3}{9}\right) }{\left(  \bar{v}(s) 
 + \frac{s^2}{3}\right)^2}  ds \right) 
\leq \frac{\bar{C}_{\ast} H(u)  }{\left(  1 + \left\vert u^{\ast} \right\vert^5 \right)  } \,.
\end{equation}

Notice that (\ref{W3E7}) holds for $u=u^{\ast}$ because Lemma~\ref{Lemma2} (\ref{W3E2a}) implies that $|v_n-\bar{v}| \leq\varepsilon$ at $u=u^{\ast}$ for $n$ sufficiently large. 
Therefore, we can extend it to some interval contained in $u>u^{\ast}$ by means of a continuation
argument. In particular, if (\ref{W3E7}) is satisfied then also 
\[
|v_n - \bar{v}| \leq\frac{1}{2} \left(  \bar{v}+\frac{u^2}{3}\right) 
\]
is satisfied, and hence 
\[
|f(u)| \leq\frac{2(\eta_{a_n})^9}{L_n^6} \frac{(\varphi_n)^3}{ \left(  \bar{v}+\frac{u^{2}}{3}\right) }
+\frac{2 \left(  1 + \frac{5 |u|^3}{9}\right) (v_n-\bar{v})^2}{ \left(  \bar{v}+\frac{u^2}{3}\right)^3 }\,.
\]
Then, using (\ref{W3E8}) and as long as (\ref{W3E7}) holds we 
can further estimate $f(u)$ as 
\[
| f(u) | \leq\frac{2(\eta_{a_n})^9}{L_n^6}    \frac{(\varphi_n(u^{\ast}))^3}{(G(u^{\ast}))^3}    \frac{(G(u))^3}{\left(\bar{v}+\frac{u^2}{3}\right)}
+2 C_1^2 \varepsilon^2  \frac{\left(  1+\frac{5 |u|^3}{9}\right)}{\left(\bar{v}+\frac{u^2}{3}\right)^3}
\frac{( H(u))^2}{\left(  1 + |u^{\ast}|^{10}\right)  }\,.
\]
Using this and (\ref{W3E6b}) in (\ref{W3E6a}) we obtain
\begin{align}
 |(v_n-\bar{v})(u)|  & \leq   \frac{2 \varepsilon  \bar{C}_{\ast} H(u)}{\left(  1+ |u^{\ast}|^5\right) }
+\frac{2(\eta_{a_n})^9}{L_n^6} \frac{(\varphi_n(u^{\ast}))^3}{(G(u^{\ast}))^3} 
\int_{u^{\ast}}^u \frac{H(u)}{H(\sigma)} \frac{(G(\sigma))^3}{\left(\bar{v}(\sigma)  
+\frac{\sigma^2}{3}\right)} d\sigma + \nonumber\\
&  + 2C_1^2 \varepsilon^2 \int_{u^{\ast}}^{u}\frac{H(u)}{H(\sigma)} 
\frac{\left(  1+\frac{5|\sigma|^3}{9}\right)}{\left( \bar{v}(\sigma)+\frac{\sigma^2}{3}\right)^3}
\frac{(H(\sigma))^2}{\left(1 + |u^{\ast}|^{10}\right) } d\sigma\,.\label{W4E2}
\end{align}
Now, using (\ref{W4E3}) and (\ref{W4E4}) in (\ref{W4E2}) gives 
\begin{align}
| (v_n -\bar{v})(u) |& \leq
 \frac{2C^{\ast}\varepsilon H(u)}{\left(  1+|u^{\ast}|^5\right)  }
+C\frac{(\eta_{a_n})^9}{L_n^6}\frac{(\varphi_n(u^{\ast}))^3}{(G(u^{\ast}))^3} H(u)\left(  1+| u^{\ast}|^3\right)
+ \nonumber\\&  
C\frac{(\eta_{a_n})^9}{L_n^6}\frac{(\varphi_n(u^{\ast}))^3}{(G(u^{\ast}))^3} H(u)\left(  1+ (u)_{+}^{\frac{9}{5}}\right)  
+ C\,C_1^2 \varepsilon^2 \frac{H(u)}{\left(  1+ |u^{\ast}|^{7}\right)}
\label{W4E5}
\end{align}
where $(u)_+=u$ if $u>0$ and $(u)_+=0$ otherwise. Notice that for $u<0$, the estimate (\ref{W4E5}) implies (\ref{W3E7}) if we
 assume that $C_1> 3 C^{\ast}$ and $n$ is large and $\varepsilon$ is small. In this case the second and third terms in the right-hand side can be absorbed 
in the first by noticing that $\frac{\eta_{a_n}}{L_n}$ behaves like $\exp(-B (u^{\ast})^3)$ for a positive constant $B$ and that $\varphi_n(u^{\ast})$ behaves like 
$(u^{\ast})^3$ (see Lemma~\ref{Lemma2}), so, in particular 
\[
\frac{\eta_{a_n}^9}{L_n^6} \frac{(\varphi_n(u^{\ast}))^3}{(G(u^{\ast}))^3} \ll \frac{\varepsilon}{1 + |u^{\ast}|^8}\ll \frac{\varepsilon}{1 + |u^{\ast}|^5}
\]
for $n$ large enough, and where we also use (\ref{W4E3}). Therefore, a continuation argument implies that (\ref{W3E7}) for $u<0$. 
In order to derive the range of values of $u>0$ we only need to see the range of values for which the third term in (\ref{W4E5}) can 
be absorbed into the first one, that is
\[
\frac{(\eta_{a_n}^9)}{L_n^6}(u)_+^\frac{9}{5}< C\frac{\varepsilon}{1+|u^\ast|^5} \quad \mbox{with} \quad C>0\,.
\]
Thus taking for example, $u<\frac{L_n^3}{\eta_{a_n}^5}$ for $n$ large enough the estimates holds (absorbing a factor $L_n^\frac{1}{3}$ in order 
to control $\varepsilon$ factor).
\end{proof}

In terms of the variable $u$ and the transformation (\ref{W2E9}), the balance that gives the maximum size of the next matching region
 corresponds to the one for which the second and third term in the equation for $\frac{dv_n}{du}$ in (\ref{W3E2}) are of the same order 
(that is $\frac{5}{3}u v_n\sim \frac{(\eta_{a_n})^9}{L_n^6} (\varphi_n)^3$), thus $u^\ast<u \ll \frac{L_n^{10}}{\eta_n^{15}}$ as $n\to \infty$. 
On the other hand $u^{\ast\ast}=\frac{1}{2} \frac{L_n^3}{(\eta_{a_n})^5} < \frac{L_n^{10}}{(\eta_{a_n})^{15}}$ for $n$ large enough.

As a next step we obtain approximations for the function $\bar{\Phi}_n(\eta_n)$ 
and the derivatives $\frac{d\bar{\Phi}_n(\eta_n)}{d\eta_n}$, $\frac{d^2\bar{\Phi}_n(\eta_n)}{d\eta_n^2}$ 
to the most right value of $\eta_n$ in the boundary layer region where the function 
$\bar{\Phi}_n$ is small. More precisely, this most right value of $\eta_n$ will be defined by means of 
$u^{\ast\ast}$.

\begin{lemma}\label{ApprDer}
Let $\xi_{\ast\ast}=\frac{2^{\frac{2}{5}} }{K_n} \left(\frac{K_+}{K_-}\right)^{\frac{4}{3}} (u^{\ast\ast})^{\frac{3}{5}}$. 
For any $\delta>0$, there exists $n_0=n_0(\delta)$ such that for $n\geq n_0 (\delta)$, there exist 
$\eta_n^{\ast\ast}\in\left[  \eta_{a_n} + \frac{(\eta_{a_n})^4 \xi_{\ast\ast}}{L_n^3}(1-\delta), 
\eta_{a_{n}}+\frac{(\eta_{a_n})^4 \xi_{\ast\ast}}{L_n^3}(1+\delta)  \right]$ 
such that:
\begin{align}
\left| \bar{\Phi}_n(\eta_n^{\ast\ast})-\frac{A_n L_n^3}{(\eta_{a_n})^5}\frac{ (\eta_n^{\ast\ast}-\eta_{a_n})^2}{4}\right|  
& \leq\frac{\delta L_n^3}{(\eta_{a_n})^5}\frac{ (\eta_n^{\ast\ast}-\eta_{a_n})^2}{4}\nonumber\\
\left| \frac{d\bar{\Phi}_n (\eta_n^{\ast\ast})}{d\eta_n}-\frac{A_n L_n^3 (\eta_n^{\ast\ast}-\eta_{a_n})  }{2(\eta_{a_n})^5}\right| 
 & \leq\frac{\delta L_n^3 (\eta_n^{\ast\ast}-\eta_{a_n})  }{(\eta_{a_n})^5} \nonumber\\
\left| \frac{d^2\bar{\Phi}_n(\eta_n^{\ast\ast})}{d\eta_n^2} -\frac{A_n L_n^3}{2(\eta_{a_n})^5}\right|  
& \leq\frac{\delta L_n^3}{(\eta_{a_n})^5} 
\label{W6E7}
\end{align}
where $A_n:=\left(\frac{K_-}{K_+}\right)^{\frac{5}{3}}K_n^5$.
\end{lemma}

\begin{remark}
The coefficient $\Gamma$ in Theorem~\ref{matchsolution} can be approximated from (\ref{W6E7}) for $n$ large enough as 
$\Gamma\sim A_n/4$. In particular this is consistent with relation $\Gamma\propto K^5$ (which follows by a scaling argument 
in (\ref{ecuacionminima}) and the behaviours (\ref{lin:match}) and (\ref{para:match}).
\end{remark}

\begin{proof}
Notice that we can assume that $z_n$ in (\ref{W2E9}), and then also $\zeta_n$ 
are functions of $u$ for $u^{\ast} < u < u^{\ast\ast}$. 
Combining (\ref{W2E9}), (\ref{W3E2}) and using also $\zeta_n=-R_{M,n}$ 
at $u=u^{\ast}$ we obtain:
\[
\zeta_n + R_{M,n}=\int_{u^{\ast}}^u \frac{\varphi_n^{\frac{4}{3}}(s)}{v_n(s) + \frac{u^2}{3}}ds\,.
\]
Clearly, from (\ref{W3E2a}) and (\ref{A1E2}), $\frac{\varphi_n (u^{\ast})}{G(u^{\ast})  }$ is uniformly bounded by some constant $C>0$. Then,
using Lemma~\ref{BL} (estimates (\ref{W3E7}) and (\ref{W3E8})) and (\ref{W3E6c}), 
\begin{align}
& \left| \zeta_n + R_{M,n} - \left(\frac{ \varphi_n (u^{\ast}) }{ G(u^{\ast}) } \right)^{\frac{4}{3}}
\int_{u^{\ast}}^u \frac{ \left(  G(s) \right)^{\frac{4}{3}} }{\bar{v}(s) + \frac{u^2}{3}}ds\right| \\
&  \leq\frac{C\, C_1 \varepsilon}{ (1 + |u^{\ast}|^2 )  }\int_{u^{\ast}}^u\frac{(G(s))^{\frac{4}{3}}}{1+s^2}ds 
+ \frac{C\,C_1 \varepsilon}{ ( 1 + |u^{\ast}|^5)  }\int_{u^{\ast}}^u\frac{( G(s))^{\frac{4}{3}}H(s)  }{(1+s^2)^2}ds
\quad  \mbox{for}\quad u < u^{\ast\ast}\,. \label{W4E8}
\end{align}
Using also the global estimates (\ref{W4E3}) and (\ref{W4E4}) of Lemma~\ref{GH}, we further get that
\begin{align}
& \left| \zeta_n + R_{M,n} - \left(  \frac{ \varphi_n(u^{\ast})  }{ G(u^{\ast}) } \right)^{\frac{4}{3}}
\int_{u^{\ast}}^0  \frac{ (G(s))^{\frac{4}{3}} }{ \bar{v}(s)  + \frac{u^2}{3} } ds
-\left(  \frac{\varphi_n (u^{\ast})  }{ G(u^{\ast})  }\right)^{\frac{4}{3}}
\int_0^u \frac{(G(s))^{\frac{4}{3}} }{ \bar{v}(s)  +\frac{u^2}{3} }ds\right| \nonumber\\
&  \leq C\, C_1\varepsilon \left( 1+ |u^{\ast}| \right) +\frac{C\, C_1 \varepsilon}{( 1+ |u^{\ast}|^2)}\left(1 + (u)_{+}^{\frac{3}{5}}\right)
\quad \mbox{for} \quad u < u^{\ast\ast} \label{W5E1}
\end{align}
where we have split the integral on the left hand side of (\ref{W4E8}).

Now, using (\ref{W3E2a}), (\ref{W3E1}), (\ref{A1E2}) and (\ref{W3E6c}), we obtain: 
\begin{equation}\label{W5E2}
\left| R_{M,n}-\left(\frac{\varphi_n (u^{\ast})}{G(u^{\ast})}\right)^{\frac{4}{3}}
\int_{u^{\ast}}^0 \frac{ (G(s))^{\frac{4}{3}}}{\bar{v}(s) + \frac{u^2}{3} }ds\right| \leq\delta R_{M,n}
\end{equation}
where $\delta>0$ can be made arbitrarily small if $\varepsilon$ is small and $n$ is large. 
On the other hand, using (\ref{W3E1}), (\ref{A1E1}) and (\ref{W3E6c}), we obtain: 
\begin{equation}\label{W5E3}
\left| \left(\frac{\varphi_n(u^{\ast})}{G(u^{\ast})}\right)^{\frac{4}{3}}
\int_0^u \frac{(G(s))^{\frac{4}{3}}}{\bar{v}(s)+\frac{u^2}{3}} ds
-\frac{2}{K_n^4}\left(  \frac{K_+}{K_-}\right)^{\frac{4}{3}}(u)^{\frac{3}{5}}\right| 
\leq\frac{\delta}{2} u^{\frac{3}{5}} \quad u>0\,,
\end{equation}
where $\delta$ can be made small choosing $\varepsilon$ small and $n$ large. 
Combining now (\ref{W5E1}), (\ref{W5E2}) and (\ref{W5E3}) we obtain, for small 
$\varepsilon$ and large $n$:
\[
\left| \zeta_n - \frac{2}{K_n^4}\left(  \frac{K_+}{K_-}\right)^{\frac{4}{3}}(u)^{\frac{3}{5}}\right| 
\leq\delta\left(R_{M,n} + (u)^{\frac{3}{5}}\right)  +C\, C_1 \varepsilon (1 + |u^{\ast}|)  
\quad \mbox{for} \quad u< |u^{\ast\ast}|\,.
\]
In particular, recalling that $u^{\ast\ast}\sim \exp{B (u^{\ast})^3}$ for some positive $B$, this gives
\begin{equation}\label{W5E7}
\left| \zeta_n-\frac{2}{K_n^4}\left(\frac{K_+}{K_-}\right)^{\frac{4}{3}}(u^{\ast\ast})^{\frac{3}{5}} \right| 
\leq\delta (u^{\ast\ast})^{\frac{3}{5}}
\end{equation}
with $\delta$ small if $\varepsilon$ is small and $n$ large.

On the other hand, using (\ref{W3E8}), (\ref{A1E1}), and similarly using 
(\ref{W3E6c}), (\ref{W4E4}) and (\ref{W3E7}), we obtain 
\begin{equation}\label{W5E5}
\left| \varphi_n (u^{\ast\ast})  -\frac{K_+}{K_-}\frac{(u^{\ast\ast})^{\frac{6}{5}}}{K_n^3}\right|
\leq \delta (u^{\ast\ast})^{\frac{6}{5}} \quad\mbox{and}\quad
\left| v_n (u^{\ast\ast})  - \frac{(u^{\ast\ast})^2}{2}\right| \leq \delta (u^{\ast\ast})^2\,,
\end{equation}
with $\delta$ small for $\varepsilon$ small and $n$ large. 
Translating these behaviours by means of (\ref{W2E9}), we finally obtain
\[
\left| \frac{d\varphi_n}{d\zeta_n} - \left(\frac{K_-}{K_+}\right)^{\frac{1}{3}} K_n (u^{\ast\ast})^{\frac{3}{5}} \right|
  \leq\delta (u^{\ast\ast})^{\frac{3}{5}}\,,\quad
\left| \frac{d^2\varphi_n}{d\zeta_n^2}-\frac{1}{2}\left(\frac{K_-}{K_+}\right)^{\frac{5}{3}}K_n^5\right| 
   \leq \delta\quad \mbox{at}\quad u=u^{\ast\ast}\,.
\]
These estimates, as well as (\ref{W5E5}), can be written in terms of the variable $\zeta_n$ by using (\ref{W5E7}), namely,
\[
\left| \varphi_n - K_n^5 \left(\frac{K_-}{K_+}\right)^{\frac{5}{3}} \frac{\zeta_n^2}{4}\right| 
\leq \delta \zeta_n^2\,,\quad \left|\frac{d\varphi_n}{d\zeta_n}-K_n^5\left(\frac{K_-}{K_+}\right)^{\frac{1}{3}}\frac{\zeta_n}{2}\right| 
\leq\delta\zeta_n\,,\quad \left| \frac{d^2\varphi_n}{d\zeta_n^2}-\frac{1}{2}\left(  \frac{K_-}{K_+}\right)^{\frac{5}{3}}K_n^5\right| 
\leq\delta
\]
at some $\zeta_n \in \left[  \frac{2(1 -\delta)}{K_n} \left(\frac{K_+}{K_-}\right)^{\frac{4}{3}} (u^{\ast\ast})^{\frac{3}{5}},\frac{2(1+\delta)}{K_n}\left(\frac{K_+}{K_-}\right)^{\frac{4}{3}}(u^{\ast\ast})^{\frac{3}{5}}\right]$. Observe that $K_n$ can be bounded above and below by 
constants or order one, independently on the choice of $a_n$. The result (\ref{W6E7}) follows now from (\ref{defphi}).
\end{proof}

We now approximate $\bar{\Phi}_n(\eta_n)$ for $\eta_n \geq\eta_n^{\ast\ast}$. 
To this end we define the following polynomials that approximate, as we shall see, the functions 
$\bar{\Phi}_n$ for $n$ large enough in the appropriate intervals. Thus we let
\[
Q_n(Y)=\left[\frac{A_n}{4} B_n^2 + \frac{A_n \,B_n}{2} Y + \frac{A_n}{4}Y^2 - \frac{Y^3}{6}\right]
\,,\quad B_n=\frac{\xi_{\ast\ast}}{L_n^6}(\eta_{a_n})^9\,, 
\]
where $A_n$ is as in Lemma~\ref{ApprDer}. The definition of $\xi_{\ast\ast}$ (cf. Lemma \ref{ApprDer}) as well as 
(\ref{W2E8}) and (\ref{W3E2a}) imply that $\lim_{n\to\infty}B_n=0$. The next result now follows:
\begin{lemma}\label{lastmatch} 
For $\delta>0$, there exists $n_0(\delta)$ such that for $n\geq n_0(\delta)$ the function $\bar{\Phi}_n(\eta_n)$ 
is increasing in an interval $\eta_n\in (\eta_n^{\ast\ast},\eta_n^{\ast\ast} + \bar{\eta}_n)$, where  
$\bar{\eta}_n\in\left[\left(  \frac{K_-}{K_+}\right)^{\frac{5}{3}}\left(  \frac{3}{2}\right)^{\frac{25}{3}}\frac{L_n^3}{(\eta_{a_n})^5}(1-\delta),
\left(  \frac{K_-}{K_+}\right)^{\frac{5}{3}} \left(  \frac{3}{2}\right)^{\frac{25}{3}}\frac{L_n^3}{(\eta_{a_n})^5}(1+\delta)\right]$. 
Moreover, defining $L_{n+1}=\Phi(\tau_{n+1}^+)$ we obtain:
\begin{equation}\label{W7E1}
\left| L_{n+1}-\frac{Q_n(Y_{\ast,n})}{(\eta_{a_n})^{15}}L_n^{10}\right| 
\leq\frac{\delta Q_n(Y_{\ast,n})}{(\eta_{a_n})^{15}} L_n^{10}
\end{equation}
where $Y_{\ast,n}=\frac{1}{2}\left(A_n+\sqrt{A_n^2+4A_n \,B_n }\right)$,  
and also:
\begin{equation}\label{W7E2}
\left| \tau_{n+1}^+ -\tau_n^+ - \eta_n^{\ast\ast} L_n^\frac{1}{3}
-Y_{\ast,n}\frac{L_n^{\frac{10}{3}}}{(\eta_{a_n})^5}\right|
\leq \delta Y_{\ast,n}\frac{L_n^{\frac{10}{3}}}{(\eta_{a_n})^5}\,.
\end{equation}
\end{lemma}
\begin{proof}
First, we define $J_n=\frac{L_n^3}{(\eta_{a_n})^5}$, for simplicity of notation, and observe that 
$\lim_{n\to \infty}J_n=\infty$. Integrating (\ref{W1E1}) we obtain
\begin{equation}\label{W6E3}
\bar{\Phi}_n(\eta_n)= P(\eta_n)
+ \frac{1}{L_n^3}\int_{\eta_n^{\ast\ast}}^{\eta_n}\int_{\eta_n^{\ast\ast}}^{s_1}\int_{\eta_n^{\ast\ast}}^{s_2}
\frac{1}{(\bar{\Phi}_n(s_3))^3} \, ds_3\,ds_2\,ds_1
\end{equation}
with
\[
P(\eta_n)=\bar{\Phi}_n(\eta_n^{\ast\ast})+\frac{d\bar{\Phi}_n(\eta_n^{\ast\ast})}{d\eta_n}\left(\eta_n-\eta_n^{\ast\ast}\right)  
+\frac{1}{2}\frac{d^2\bar{\Phi}_n(\eta_n^{\ast\ast})}{d\eta_n^2}(\eta_n-\eta_n^{\ast\ast})^2
-\frac{(\eta_n-\eta_n^{\ast\ast})^{3}}{6}\,.
\]
As a direct consequence of Lemma~\ref{ApprDer} one has that
\begin{equation}\label{W6E2}
\left|P(\eta_n)- J_n^3 Q(Y)  \right| +\left|\frac{d}{d\eta_n} (P(\eta_n)-J_n^3 Q(Y))\right| 
+\left|\frac{d^2}{d\eta_n^2}(P(\eta_n)  -J_n^3 Q(Y) )  \right| 
\leq \delta J_n^3 Q(Y) 
\end{equation}
with $\eta_n^{\ast\ast}\leq\eta_n$,  $Y\leq (1+\delta) Y_{\ast,n}$, if $n$ 
is sufficiently large and where $\eta_n-\eta_n^{\ast\ast}=J_n Y$.

We can now estimate the effect of the last term in (\ref{W6E3}). To this end 
we use the type of continuation argument that we had used repeatedly. Now our goal 
is to show that:
\begin{equation}\label{W6E4}
\bar{\Phi}_n(\eta_n) \geq \frac{J_n^3 Q(Y)}{8}\,.
\end{equation}
This inequality is satisfied for $\eta_n=\eta_n^{\ast\ast}$. On the other hand, as long as this inequality holds, we have, using the change 
of variables $\eta_n-\eta_n^{\ast\ast}=J_n Y$, that
\[
\frac{1}{L_n^3}
\int_{\eta_n^{\ast\ast}}^{\eta_n}\int_{\eta_n^{\ast\ast}}^{s_1}\int_{\eta_n^{\ast\ast}}^{s_2}\frac{1}{(\bar{\Phi}_n(s_{3}))^{3}}ds_3\,ds_2 \, ds_1
\leq\frac{2^9}{L_n^3 J_n^6} \int_0^Y \int_{0}^{Y_1} \int_{0}^{Y_2}\frac{1}{(Q(Y_3))^3} dY_3\, dY_2\, dY_1\,.
\]
If $0<Y_3\leq Y_{\ast,n}(1 + \delta)$ then $Q(Y_3) \geq C\left(\frac{A_n}{4}\,B_n^2+\frac{A_n}{2}\,B_n Y_3 +\frac{A_n}{4} Y_3^2\right)$. 
Then, using the change of variable $Y=B_n Z$ we obtain:
\begin{eqnarray*}
\frac{1}{L_n^3}\int_{\eta_n^{\ast\ast}}^{\eta_n} \int_{\eta_n^{\ast\ast}}^{s_1}\int_{\eta_n^{\ast\ast}}^{s_2}\frac{1}{(\bar{\Phi}_n(s_3))^3}ds_3 \, ds_2\,  ds_1
\\
\leq\frac{C}{L_n^3 J_n^6 B_n^3}\int_0^{Y/B_n}\int_0^{Z_1}\int_0^{Z_2}\frac{1}{\left(\frac{A}{4}+\frac{A}{2}Z_3 + \frac{A}{4}Z_3^2\right)^3} dZ_3\,dZ_2\,dZ_1    
\leq\frac{C(Y)^2}{L_n^3 J_n^6 B_n^5}\,,
\end{eqnarray*}
for $n$ large enough. Here the first integral in the integrand can be estimated by a constant using that $B_n\to 0$ as $n\to\infty$. 
Analogous estimates can be proved for the first two derivatives of this
integral with a similar argument.

We observe that $\frac{C}{L_n^3 J_n^6 B_n^5}\ll J_n^3$ as $n\to\infty$, 
since this is equivalent to $C\ll L_n^3 J_n^9 B_n^5=(\xi_{\ast\ast})^5$ 
as $n\to\infty$ and this is satisfied due to the definition of $\xi_{\ast\ast}$. 
Therefore (\ref{W6E4}) holds for $0\leq Y\leq (1+\delta) Y_{\ast,n}$. 
It then follows from (\ref{W6E3}), (\ref{W6E2}) that:
\begin{equation}\label{W6E5}
\left| \bar{\Phi}_n(\eta_n)-J_n^3 Q(Y)  \right|
+\left| \frac{d}{d\eta_n}(\bar{\Phi}_n(\eta_n)-J_n^3 Q(Y))  \right| 
+\left| \frac{d^2}{d\eta_n^2}(\bar{\Phi}_n(\eta_n)-J_n^3 Q(Y))\right|
\leq\delta J_n^3 Q(Y) \,,
\end{equation}
with $\eta_n^{\ast\ast}\leq\eta_n$ and $Y\leq (1+\delta)  Y_{\ast,n}$.

It then follows the existence of $\bar{\eta}_n>0$ as in the statement of the Lemma, and such that
 $\frac{d}{d\eta_n}(\bar{\Phi}_n(\eta_n^{\ast\ast}+\bar{\eta}_n))=0$, 
and that $\frac{d}{d\eta_n} (\bar{\Phi}_n(\eta_n))>0$ for $\eta_n\in (\eta_n^{\ast\ast},\eta_n^{\ast\ast}+\bar{\eta}_n)$. 
Moreover, (\ref{W6E5}) implies (\ref{W7E1}) and (\ref{2order1:scale}) yields (\ref{W7E2}).
\end{proof}

\smallskip
\begin{proof}[Proof of Theorem \ref{mainoscillation}]The asymptotics (\ref{W1E8}) is just a 
consequence of (\ref{W7E1}) that can be solved recursively arguing as in the 
formal derivation of (\ref{W1E8}). On the other hand (\ref{W1E8a}) is a 
consequence of (\ref{2order1:scale}), (\ref{defphi}) and Lemma \ref{BL}. 
Finally (\ref{secder}) is a consequence of (\ref{W6E5}). Indeed, notice that 
this last formula implies
\begin{equation}\label{W7E3}
a_{n+1} = -(1+O(\delta))  \frac{Q^{\prime\prime}(Y_{\ast})}{(Q(Y_{\ast}))^{\frac{1}{3}}}\,,
 \quad Y_{\ast,n}=\frac{1}{2} (A_n+\sqrt{A_n^2+4 A_n B_n})\,.  
\end{equation}
Since $B_n\to 0$, we have $(Y_{\ast,n}-A_n)\to 0$ as $n\to\infty$. Therefore $
0<\varepsilon_0\leq a_{n+1}\leq\frac{1}{\varepsilon_{0}}$ for $n$ large. Then
\[
\frac{Q^{\prime\prime} (Y_{\ast,n})}{ (Q(Y_{\ast,n}))^{\frac{1}{3}}}
\to\left(\frac{3}{2}\right)^{\frac{1}{2}} \quad\mbox{as}\quad n\to\infty
\]
whence (\ref{W7E3}) implies, since $\delta$ can be assumed to be arbitrarily 
small for large $n$, that $a_{n+1}\to a$,
with $a$ as in (\ref{a:value}), and the result follows.
\end{proof}

\smallskip

\noindent
{\bf Acknowledgements} The authors wish to thank Marco Fontelos, John King, Andreas M\"unch, Tim Myers and Barbara Wagner, for their helpful comments and for pointing out many relevant references. C.~M. Cuesta also acknowledges the support of the Spanish Ministry of Science for the starting grant part of a {\it Ram\'on y Cajal} fellowship. J.~J.~L. Vel\'azquez was partially supported by the DGES grant MTM2010-16457.

\gdef\thesection{A1} \renewcommand{\theequation}{A1.\arabic{equation}}\setcounter{equation}{0}
\section*{Appendix 1. The operator $(\nabla\mathbf{F} + \nabla\mathbf{F}^{T}) \mathbf{N}$ in curvilinear coordinates}
As above, vectors components in Cartesian coordinates are denoted by sub-indexes $x$ and 
$y$ respectively, and when expressed in curvilinear coordinates the indexes 
$1$ (tangential to the substrate) and $2$ (normal to the substrate) are used. For a general vector $\mathbf{F}$, thus with 
$\mathbf{F}= F_{x} \mathbf{e}_{x} + F_{y} \mathbf{e}_{y} = F_{1} \mathbf{t} +
F_{2} \mathbf{n}$, we compute  $(\nabla\mathbf{F} + \nabla\mathbf{F}^{T}) \mathbf{N}$. This is done, for example,  
using the expression in curvilinear coordinates for linear differential operators given in e.g. \cite{acheson}, 
and one obtains, writing separately the normal and tangential components:
\begin{align*}
((\nabla\mathbf{F} + \nabla\mathbf{F}^{T}) \mathbf{N})\cdot\mathbf{N} =\\
2\frac{\partial h}{\partial s}\frac{\frac{\partial h}{\partial s} t_{1}
-(1-kh)t_{2} }{\left(  (1 - k\, h)^{2}  
+\left(  \frac{\partial h}{\partial s}\right)^{2}\right)  (1-k\,h)} \left(  \frac{\partial F_{1}}{\partial s}
- kF_{2} \right)  +2(1-k\,h) \frac{\frac{\partial h}{\partial s} t_{2}
+(1-k\,h)t_{1}}{(1 - k\, h)^{2} 
+ \left(  \frac{\partial h}{\partial s}\right)^{2}} \frac{\partial F_{2}}{\partial d}\\
- \frac{ 2 \frac{\partial h}{\partial s} (1-k\,h) t_{1} + 
(\left(\frac{\partial h}{\partial s}\right)^{2} 
- (1 - k\, h)^{2}) t_{2}}{\left( (1- k\, h)^{2} 
+ \left(  \frac{\partial h}{\partial s}\right) ^{2}\right)(1-k\,h) } 
\left(  \frac{\partial F_{2}}{\partial s} + kF_{1} + (1-k\,h)\frac{\partial F_{1}}{\partial d}\right)
\end{align*}
and
\begin{align*}
((\nabla\mathbf{F} + \nabla\mathbf{F}^{T}) \mathbf{N})\cdot\mathbf{T}=\\
\frac{ 2\frac{\partial h}{\partial s}(1-k\,h) t_{2} + 2((1-k\,h)^{2}
-(\frac{\partial h}{\partial s})^{2})t_{1} }{\left(  (1 - k\, h)^{2} 
+ \left(\frac{\partial h}{\partial s}\right)^{2}\right)  (1-k\,h)} 
\left(\frac{\partial F_{2}}{\partial s}+ kF_{1}+ (1-k\,h)\frac{\partial F_{1}}{\partial d}\right) \\
-2 \frac{\partial h}{\partial s}\frac{(1-k\,h)t_{1} + \frac{\partial h}{\partial s} t_{2}}{\left(  (1 - k\, h)^{2} 
+ \left(  \frac{\partial h}{\partial s}\right) ^{2} \right)  (1-k\,h)} \left(  \frac{\partial F_{1}}{\partial s}
- kF_{2}+ (1-k\,h)\frac{\partial F_{2}}{\partial d} \right)\,.
\end{align*}

\gdef\thesection{A2} \renewcommand{\theequation}{A2.\arabic{equation}}\setcounter{equation}{0}
\section*{Appendix 2. Analysis of the equation (\ref{SEphi0})}

In this section we concentrate on the study of equation (\ref{SEphi0}). In 
order to clarify them matching conditions needed in the inner regions where 
$\Phi\sim0$ that are part of the oscillatory solutions to (\ref{WilsonJones}). 
The inner variables used in this inner regions are $\zeta_{n}$, and 
$\varphi_{n}$ and so we rewrite (\ref{SEphi0}) in this variables but dropping 
the subindex $n$. The main result of this appendix is: 

\begin{theorem}\label{matchsolution} There exists a unique solution of
\begin{equation}\label{ecuacionminima}
\frac{d^{3}\varphi}{d\zeta^{3}}=\frac{1}{\varphi^{3}}
\end{equation}
with the matching condition:
\begin{equation}\label{lin:match}
\varphi(\zeta)  \sim -K\zeta+o(1)  \quad \mbox{as}\quad \zeta\to -\infty\,.
\end{equation}
Moreover, the asymptotics of $\varphi(\zeta)$ for large $\zeta$ is
given by:
\begin{equation}\label{para:match}
\varphi(\zeta)  \sim\Gamma\zeta^2\quad \mbox{as}\quad \zeta\to\infty 
\quad \text{for some}\quad \Gamma>0\,.
\end{equation}
Finally, the exists a unique solution of (\ref{para:match}) with matching condition
\[
\varphi(\zeta)  \sim -\tilde{K}\zeta+o(1)  \quad \mbox{as}\quad \zeta\to \infty\,\quad \text{for some}\quad \tilde{K}>0\,.
\]
It also satisfies that there exists a finite $\zeta^\ast$ such that
\begin{equation}\label{back:zero}
\varphi(\zeta)\to 0 \quad \mbox{as}\quad \zeta\to (\zeta^\ast)^{+}\,.
\end{equation}
All other solutions satisfy (\ref{para:match}) for increasing $\zeta$, and, for decreasing $\zeta$, either (\ref{back:zero}) or
\[
\varphi(\zeta)  \sim\tilde{\Gamma}\zeta^2 \quad \mbox{as}\quad \zeta\to -\infty 
\quad \text{for some}\quad \tilde{\Gamma}>0\,
holds.
\]
\end{theorem}

We remark that in this paper we only need the existence of the solution satisfying (\ref{lin:match}) and (\ref{para:match}), 
the rest of the behaviours are stated for completeness and because they will be used in \cite{CV2}.

In order to prove this theorem, we reduce equation (\ref{ecuacionminima}) to a 
second order autonomous system of ODEs, as follows
\begin{equation}\label{SEphi0-trans}
\frac{d\varphi}{d\zeta}=\varphi^{-\frac{1}{3}}\,u\,, \ \frac{d^{2}\varphi}{d\zeta}
=\varphi^{-\frac{5}{3}}\,v \,,\ d\zeta=\varphi^{\frac{4}{3}}dz\,,
\end{equation}
giving the system
\begin{equation}\label{phi0system}
\frac{d\varphi}{d\zeta}=u\,\varphi\, , \quad
\frac{du}{dz}=v+\frac{1}{3}u^{2}\,, \quad
\frac{dv}{dz}=1+\frac{5}{3}v\,u\,.
\end{equation}

We now prove several lemmas that give the behaviour of the orbits of this system. 

The last two equations in (\ref{phi0system}) can be studied independently by 
means of a phase-plane analysis. The isoclines are 
$\Gamma_{1}=\{(u,v):\ v+\frac{1}{3}u^{2}=0\}$, that has $du/dz=0$, and 
$\Gamma_{2}=\{(u,v):\ 1+\frac{5}{3}v\,u=0\}$, that has $dv/dz=0$. 
The only critical point is $p_{e}=(u_{e},v_{e})=((9/5)^{\frac{1}{3}},-(1/3)(9/5)^{\frac{2}{3}})$, 
and linearisation gives two complex eigenvalues with positive real part, namely 
$\lambda_{\pm}=\frac{1}{6}\left(  \frac{9}{5}\right)^{\frac{1}{3}}\left(7\pm\sqrt{11}i\right)$.

We distinguish five regions in the phase-plane, the ones separated by the 
isoclines. These are depicted in 
Figure~\ref{phase-plane} where direction field is also shown in each of these 
regions. 
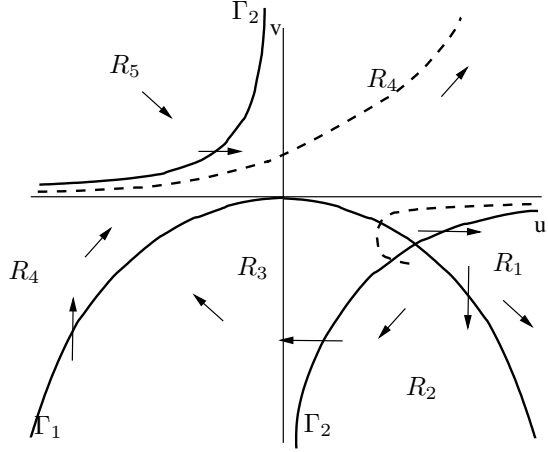
\begin{figure}[th] 
\begin{center}
\input{phase-plane-small.pstex_t}
\end{center}
\caption{Phase portrait associated to (\ref{phi0system}) showing the direction 
field. The solid lines represent the isoclines and the dashed ones the separatrices.}
\label{phase-plane}
\end{figure}
Standard arguments imply that any orbit on the phase-plane eventually crosses 
the isoclines into the region $R_{4}$ forwardly in $z$. If, however, an orbit 
has $(u,v)\in R_{4}$ at some value of $z$, it is possible to discern from 
which of the regions is coming from for smaller values of $z$ by identifying 
the separatrices of the system. We have the following result.

\begin{lemma}[Separatrices]\label{separa}
\begin{enumerate}
\item There exists a unique orbit $v=\bar{v}(u)$ in the phase-plane associated 
to system (\ref{phi0system}) that is contained in $R_{4}$ for all $u\in\mathbb{R}$. 
Moreover, $\bar{v}(u)$ has the following asymptotic behaviour
\begin{equation}\label{W3E6c}
\bar{v}(u)  =\frac{u^{2}}{2} + O\left(u^{\frac{4}{5}}\right)
\quad \mbox{as}\quad u\to\infty\,,\quad \bar{v}(u) = 
-\frac{1}{2u}(1+o(1))  \quad\mbox{as}\quad u\to -\infty
\end{equation}

\item There exists a unique orbit $v=\hat{v}(u)$ in the phase-plane associated 
to system (\ref{phi0system}) that has the following asymptotic behaviour 
\begin{equation}\label{W3E6d}
\hat{v}(u)= -\frac{1}{2u}\left(  1+o\left(  1\right)  \right)
\quad \mbox{as}\quad u\to +\infty
\end{equation}
and
\begin{equation}\label{W3E6e}
|(u,\hat{v})-(u_{e},v_{e})|\leq C e^{\mbox{Re}(\lambda_{+})z} 
\quad\mbox{as}\quad z\to-\infty\,.
\end{equation}
\end{enumerate}
\end{lemma}

\begin{remark}
The separatrix $v=\bar{v}(u)$ is uniquely determined by the 
problem
\begin{equation}\label{W3E3}
\frac{d\bar{v}}{du}=\frac{1+\frac{5}{3}u\bar{v}}{\bar{v}+\frac{1}{3}u^{2}}\,,
\quad \bar{v}(u)  \sim - \frac{1}{2u}\quad \mbox{as}\quad u\to -\infty\,.
\end{equation}
\end{remark}

\begin{proof}
In order to capture the behaviour (\ref{W3E6d}), we 
perform the following transformation.
\begin{equation}\label{trans:hyp}
w = u\,v\,,\ \frac{d}{dr}=v\frac{d}{dz}\,.
\end{equation}
If the behaviour is as expected then the transformation is locally valid as long as $v\neq 0$, thus, 
one can integrate the second equation in (\ref{trans:hyp}). Using the equation for $u$ in (\ref{phi0system}) 
we obtain the system 
\begin{equation}\label{hyp-system}
\frac{dw}{dr} = w + 2 w^{2} + v^{3}\,, \quad
\frac{dv}{dr} =v\left(  1 + \frac{5}{3}w \right)  \,,
\end{equation}
that has critical points $(-\frac{1}{2},0)$, $(0,0)$ and $\left(-\frac{3}{5},-\left(  \frac{3}{25}\right)^{\frac{1}{3}}\right)$. 
The point $\left(  0,-\frac{1}{2}\right)$ is a saddle-point. Linearisation around this point gives a 
diagonal matrix with eigenvalues $-1$ in the direction $\binom{1}{0}$ and 
$\frac{1}{6}$ in the direction $\binom{0}{1}$. 
The point $\left(  0,0\right) $ is a (degenerate) source (thus unstable).  
Linearisation around it gives a diagonal matrix with the double eigenvalue $1$. 
The point $\left(  -\frac{3}{5},-\left(\frac{3}{25}\right)^{\frac{1}{3}}\right)$ is an unstable spiral. 
Indeed, linearisation gives the eigenvalues $-\frac{7}{10}\pm\frac{1}{10}\sqrt{11}i$. 
The phase-portrait for system (\ref{hyp-system}) is shown in Figure~\ref{hyp-phase-plane}. 
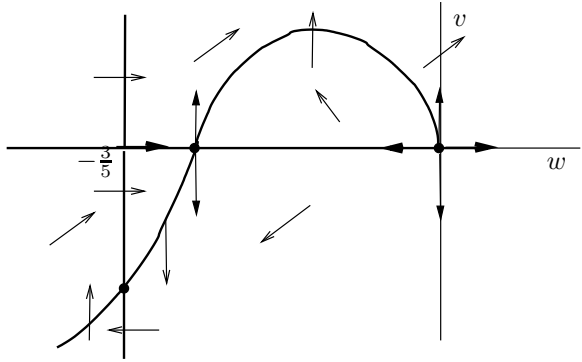
\begin{figure}[th]
\begin{center}
\input{hyperebolicbhv.pstex_t}
\end{center}
\caption{Phase portrait associated to (\ref{hyp-system}) showing the direction 
field and the eigenvectors at non-oscillatory critical points. The thick solid 
lines represent the isoclines.} 
\label{hyp-phase-plane}
\end{figure}
By using standard arguments it is easy to show that the unique orbit coming 
out of $(-1/2,0)$ tangential to $\binom{0}{-1}$ ends up at the critical point 
$\left(  -\frac{3}{5},-\left(  \frac{3}{25}\right)^{\frac{1}{3}}\right)$ 
as $r\to+\infty$ and has $v<0$ for all $r$. The unique orbit coming out of $(-1/2,0)$ tangential to 
$\binom{0}{1}$ stays unbounded in the region with $dv/dw>0$. The local behaviour of these orbits
give in the original variables give two separatrices. From the stable manifold theorem, there 
exist a $\bar{r}$ sufficiently small and a positive constant $C>0$ such that 
\[
\left|  \binom{w}{v}-\binom{-\frac{1}{2}}{0}\right|  \leq e^{\frac{1}{6}r}C 
\quad\mbox{for all}\quad r<\bar{r}\,.
\]
We then have 
\[
|v|\,,\ |u\,v+\frac{1}{2}| \leq e^{\frac{1}{6}r}C \quad\mbox{for}\quad r<\bar{r}
\]
and there are positive constants $C_{1}$ and $C_{2}$ such that 
\[
6 C\, e^{\frac{1}{6}r} + C_{1} \leq z(r)\leq6 C\, e^{\frac{1}{6}r} + C_{2} 
\quad\mbox{for}\quad r\leq\bar{r}
\]
and, in particular, there exists a constant $z_{0}$ such that 
\[
\lim_{r\to-\infty}z(r)= z_{0} \quad\mbox{and} \quad z(r)>z_{0} \quad 
\mbox{for all} \quad r\leq\bar{r}\,.
\]
Translating this behaviour into the original variables of system (\ref{phi0system}), we obtain the existence of 
a unique orbit $\bar{v}(u)$ satisfying the second asymptotic behaviour in (\ref{W3E6c}), and the existence of a unique orbit 
$\hat{v}(u)$ satisfying (\ref{W3E6d}) 
the solutions on this orbit (with $v<0$ for all $r$) have, translating the behaviour around the critical 
point $\left(  -\frac{3}{5},-\left(  \frac{3}{25}\right) ^{\frac{1}{3}}\right)$, the behaviour (\ref{W3E6e}).

We now infer the behaviour of orbits entering $R_{4}$ as $u\to+\infty$. 
Suppose that there are orbits with the behaviour $v\sim C u^{2}$ as 
$|u|\to\infty$, then $dv/du\sim2 C u$ and from (\ref{phi0system}) 
$dv/du\sim\frac{1+\frac{5C}{3}u^{2}}{Cu^{2}+\frac{1}{3}u^{2}}$, thus $C=1/2$. 
To prove that all orbits, in fact, have this behaviour, we set 
\[
V=\frac{v}{u^2},\ U=\frac{1}{u},\ U\frac{d}{dz}=\frac{d}{ds} 
\]
and the system becomes 
\begin{equation}\label{par-system}
\frac{dU}{ds} =-U\left(  V+\frac{1}{3}\right)  \,,  \quad
\frac{dV}{ds} =U^{3}+(V-2V^{2})\,.
\end{equation}
the transformation is valid locally as long as $U\neq 0$. 
This system has critical points $(  0,\frac{1}{2})$  ,$(  0,0)$ and 
$\left(  \left( \frac{5}{9}\right)^{\frac{1}{3}},-\frac{1}{3}\right)$. 
The point $\left(  0,\frac{1}{2}\right)  $ is a sink. Linearisation 
around it gives a diagonal matrix with eigenvalues $-\frac{5}{6}$ with direction 
$\binom{1}{0}$, and $-1$ with direction $\binom{0}{1}$. The point $\left(  0,0\right)$ is a saddle-point. 
Linearisation around it gives a diagonal matrix with eigenvalues $1$ with direction $\binom{0}{1}$, and $-\frac{1}{3}$ 
with direction $\binom{1}{0}$. Finally, the point $\left(  \left(  \frac{5}{9}\right)^{\frac{1}{3}},-\frac{1}{3}\right)$ 
is an unstable spiral. The phase-portrait for system (\ref{par-system}) is shown in Figure~\ref{par-phase-plane}. 
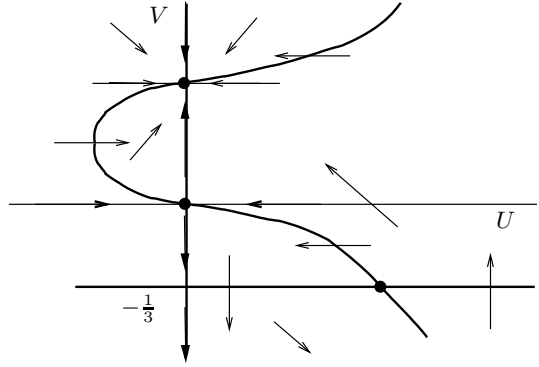
\begin{figure}[th]
\begin{center}
\input{parabolicbhv.pstex_t}
\end{center}
\caption{Phase portrait associated to (\ref{par-system}) showing the direction
 field and the eigenvectors at non-oscillatory critical points. The thick solid 
lines represent the isoclines.}
\label{par-phase-plane}
\end{figure}
The orbits entering the critical point $(  0,\frac{1}{2})$ have thus $U>0$ and $V>0$, correspond to the 
orbits associated to solutions of (\ref{phi0system}) that at a point are in $R_{4}$ with $u>0$ and $v>0$ for 
increasing $u$. The ones entering with $U<0$ correspond to orbits associated to solutions of (\ref{phi0system})
 that at  a point are in $R_5$ with $u<0$, $v>0$ for decreasing $u$. Observe that these orbits are associated 
to the slow eigenvalue $-\frac{5}{6}$ (the only orbit entering through the fast one, $-1$, is contained in the 
segment $U=0$, $V\in(0,1/2)$ and thus not correspond to solutions of (\ref{phi0system})). Arguing as previously 
we then have that there exists a sufficiently large $C>0$ such that for all orbits with $v(u=0)>0$ then   
\begin{equation}\label{par:behv} 
\left\vert \frac{v}{u^{2}}-\frac{1}{2}\right\vert \leq\frac{C}{|u|}
\quad\mbox{as}\quad u\rightarrow+\infty\,.
\end{equation}
In particular the separatrix $\bar{v}(u)$ satisfies (\ref{par:behv}). 
In order to improve the estimate, we now take in (\ref{phi0system}) $v=\frac{1}{2}u^2+v_1$, 
hence $v_1$ satisfies
\begin{equation}\label{v1:eq}
\frac{d v_1}{du} = \frac{4}{5} \frac{v_1}{u}+ R(u,v_1)\quad \mbox{with}\quad 
R(u,v_1)=
\frac{1}{u} 
\displaystyle{
\frac{\frac{1}{u}-\frac{4}{5}\frac{v_1^2}{u^2} }{\frac{5}{6}+\frac{v_1}{u^2}} 
}
\end{equation}
Let now $\bar{u}>0$ such that $(\bar{u})^3>2C$ where $C$ is as above, and let $L=2C|\bar{u}|^\frac{6}{5}$, so that, 
in particular we can write $v_1(\bar{u})\leq L C |\bar{u}|^\frac{4}{5}L$. Then as long as $u>\bar{u}$ and 
$u^{\frac{6}{5}}<(\bar{u})^\frac{6}{5}/(5C)$ then $v_1(u)<L u^{\frac{4}{5}}$. This is proved by a bootstrap argument: 
if $v_1(u)<L u^{\frac{4}{5}}$ holds, then 
\[
|R(u,v_1)|\leq\frac{6}{5}\left(\frac{1}{|u|^\frac{3}{5}}+\frac{4}{5} L^2\right)\frac{1}{|u|^\frac{7}{5}} 
\]
and hence, integrating (\ref{v1:eq}), one gets that $v_1(u)<L u^{\frac{4}{5}}$ holds provided that 
$0<u^{\frac{6}{5}}<(\bar{u})^\frac{6}{5}/(5C)$.
\end{proof}

\begin{lemma}\label{tozero-bwds} All orbits associated to solutions of 
(\ref{phi0system}) enter $R_4$. Those that are below the separatrix $\bar{v}$ come from the 
critical point $p_e$ as $z\to-\infty$. All other orbits, except for $\bar{v}$, come from the region $R_5$. 
\end{lemma}
\begin{proof} We only need to rule out the existence of periodic orbits around the critical point. 
We then construct a spiral box invariant for decreasing $z$. The direction field 
implies that such an orbit can only be contained in the second quadrant. 
Whence, we start constructing the invariant set in the region $R_3$ at a point $(u_0,v_0)$ on $\Gamma_1$ and such that 
$0<u_0<u_e$ and $v_0=-\frac{1}{3}u_0^2$. We continue by tracing the next point $(u_1,v_1)$ on $\Gamma_2$ 
vertically downwards, i.e. $u_1=u_0$, $v_1=-\frac{3}{5}\frac{1}{u_0}$. The next point $(u_2,v_2)$ on $\Gamma_3$ is traced from 
$(u_1,v_1)$ horizontally to the right;$u_2=\sqrt{-3v_1}=\frac{3}{\sqrt{5 u_0}}$, $v_2=v_1=-\frac{3}{5}\frac{1}{u_0}$,
and so on, then $u_4  =\sqrt{3\frac{\sqrt{u_0}}{\sqrt{5}}}$ and $v_4= -\frac{\sqrt{u_0}}{\sqrt{5}}$. 
Observe that both $v_0>v_4$ and $u_4>u_0$ hold (by the condition 
$u_0<u_e=\left(\frac{3}{\sqrt{5}}\right)^\frac{2}{3}$). Then, since each of the regions enclosed by the segment 
joining $(u_{i},v_{i})$ to $(u_{i+1},v_{i+1})$ for $i=0,1,2,3$ and the 
isoclines have the property that for decreasing $z$ an orbit can only scape the 
region through the isoclines, there cannot be periodic orbits. 
\end{proof}

The proof of the theorem follows by changing to the original variables, Lemma~\ref{separa} and Lemma~\ref{tozero-bwds}.

\def\cprime{$'$}

\end{document}

%% file: phase-plane-small.pstex_t
\begin{picture}(0,0)%
\includegraphics{phase-plane-small.pstex}%
\end{picture}%
\setlength{\unitlength}{3947sp}%
\begingroup\makeatletter\ifx\SetFigFont\undefined%
\gdef\SetFigFont#1#2#3#4#5{%
  \reset@font\fontsize{#1}{#2pt}%
  \fontfamily{#3}\fontseries{#4}\fontshape{#5}%
  \selectfont}%
\fi\endgroup%
\begin{picture}(3489,2844)(152,-2047)
\put(510,-1934){\makebox(0,0)[b]{\smash{{\SetFigFont{10}{12.0}{\familydefault}{\mddefault}{\updefault}{\color[rgb]{0,0,0}$\Gamma_1$}%
}}}}
\put(2855,-1690){\makebox(0,0)[b]{\smash{{\SetFigFont{10}{12.0}{\familydefault}{\mddefault}{\updefault}{\color[rgb]{0,0,0}$R_2$}%
}}}}
\put(3401,-908){\makebox(0,0)[b]{\smash{{\SetFigFont{10}{12.0}{\familydefault}{\mddefault}{\updefault}{\color[rgb]{0,0,0}$R_1$}%
}}}}
\put(1793,-937){\makebox(0,0)[b]{\smash{{\SetFigFont{10}{12.0}{\familydefault}{\mddefault}{\updefault}{\color[rgb]{0,0,0}$R_3$}%
}}}}
\put(991,338){\makebox(0,0)[b]{\smash{{\SetFigFont{10}{12.0}{\familydefault}{\mddefault}{\updefault}{\color[rgb]{0,0,0}$R_5$}%
}}}}
\put(2205,-1919){\makebox(0,0)[b]{\smash{{\SetFigFont{10}{12.0}{\familydefault}{\mddefault}{\updefault}{\color[rgb]{0,0,0}$\Gamma_2$}%
}}}}
\put(1753,677){\makebox(0,0)[b]{\smash{{\SetFigFont{10}{12.0}{\familydefault}{\mddefault}{\updefault}{\color[rgb]{0,0,0}$\Gamma_2$}%
}}}}
\put(2610,247){\makebox(0,0)[b]{\smash{{\SetFigFont{10}{12.0}{\familydefault}{\mddefault}{\updefault}{\color[rgb]{0,0,0}$R_4$}%
}}}}
\put(360,-961){\makebox(0,0)[b]{\smash{{\SetFigFont{10}{12.0}{\familydefault}{\mddefault}{\updefault}{\color[rgb]{0,0,0}$R_4$}%
}}}}
\end{picture}%

%% file: hyperebolicbhv.pstex_t
\begin{picture}(0,0)%
\includegraphics{hyperebolicbhv.pstex}%
\end{picture}%
\setlength{\unitlength}{3947sp}%
\begingroup\makeatletter\ifx\SetFigFont\undefined%
\gdef\SetFigFont#1#2#3#4#5{%
  \reset@font\fontsize{#1}{#2pt}%
  \fontfamily{#3}\fontseries{#4}\fontshape{#5}%
  \selectfont}%
\fi\endgroup%
\begin{picture}(3667,2277)(1694,-1568)
\put(5108,-355){\makebox(0,0)[lb]{\smash{{\SetFigFont{10}{12.0}{\familydefault}{\mddefault}{\updefault}{\color[rgb]{0,0,0}$w$}%
}}}}
\put(2155,-348){\makebox(0,0)[lb]{\smash{{\SetFigFont{10}{12.0}{\familydefault}{\mddefault}{\updefault}{\color[rgb]{0,0,0}$-\frac{3}{5}$}%
}}}}
\put(4520,547){\makebox(0,0)[lb]{\smash{{\SetFigFont{10}{12.0}{\familydefault}{\mddefault}{\updefault}{\color[rgb]{0,0,0}$v$}%
}}}}
\end{picture}%

%% file: parabolicbhv.pstex_t
\begin{picture}(0,0)%
\includegraphics{parabolicbhv.pstex}%
\end{picture}%
\setlength{\unitlength}{3947sp}%
\begingroup\makeatletter\ifx\SetFigFont\undefined%
\gdef\SetFigFont#1#2#3#4#5{%
  \reset@font\fontsize{#1}{#2pt}%
  \fontfamily{#3}\fontseries{#4}\fontshape{#5}%
  \selectfont}%
\fi\endgroup%
\begin{picture}(3341,2301)(994,-1747)
\put(4065,-883){\makebox(0,0)[lb]{\smash{{\SetFigFont{9}{10.8}{\familydefault}{\mddefault}{\updefault}{\color[rgb]{0,0,0}$U$}%
}}}}
\put(1885,404){\makebox(0,0)[lb]{\smash{{\SetFigFont{9}{10.8}{\familydefault}{\mddefault}{\updefault}{\color[rgb]{0,0,0}$V$}%
}}}}
\put(1712,-1425){\makebox(0,0)[lb]{\smash{{\SetFigFont{9}{10.8}{\familydefault}{\mddefault}{\updefault}{\color[rgb]{0,0,0}$-\frac{1}{3}$}%
}}}}
\end{picture}%

%% file: CuesVelazIarxiv.bbl
\begin{thebibliography}{10}

\bibitem{acheson}
D.~J. Acheson.
\newblock {\em Elementary Fluid Dynamics}.
\newblock Oxford University Press, 1990.

\bibitem{benilov}
E.~S. Benilov, S.~J. Chapman, J.~B. McLeod, J.~R. Ockendon, and V.~S. Zubkov.
\newblock On liquid films on an inclined plate.
\newblock {\em Journal of Fluid Mechanics}, 663:53--69, 2010.

\bibitem{BerettaHP}
E.~Beretta, J.~Hulshof, and L.~A. Peletier.
\newblock On an ode from forced coating flow.
\newblock {\em Journal of Differential Equations}, 130:247--265, 1996.

\bibitem{BertozziPhF}
A.~L. Bertozzi and M.~P. Brenner.
\newblock Linear stability and transient growth in driven contact lines.
\newblock {\em Physics of Fluids}, 9:530--539, 1997.

\bibitem{BertozziEJAM}
A.~L. Bertozzi, A.~M{\"u}nch, M.~Shearer, and K.~Zumbrun.
\newblock Stability of compressive and undercompressive thin film travelling
  waves.
\newblock {\em European Journal of Applied Mathematics}, 12:253--291, 2001.

\bibitem{Boatto}
S.~Boatto, L.~P. Kadanoff, and P.~Olla.
\newblock Traveling-wave solutions to thin-film equations.
\newblock {\em Physical Review E}, 48:4423--4431, 1993.

\bibitem{bretherton}
F.~P. Bretherton.
\newblock The motion of long bubbles in tubes.
\newblock {\em Journal of Fluid Mechanics}, 10:166--188, 1961.

\bibitem{CrasterMatar}
V.~Craster and O.~K. Matar.
\newblock Dynamics and stability of thin liquid films.
\newblock {\em Reviews of Modern Physics}, 81:1131--1198, 2009.

\bibitem{CV2}
C.~M. Cuesta and J.~J.~L. Vel{\'a}zquez.
\newblock Fluid accumulation in thin-film flows driven by surface tension and
  gravity (ii): Analysis of the boundary layer equation.
\newblock In preparation.

\bibitem{Eggers}
J.~Eggers.
\newblock Hydrodynamic theory of forced dewetting.
\newblock {\em Physical Review Letters}, 93(9):094502, 2004.

\bibitem{EggersII}
J.~Eggers.
\newblock Existence of receding and advancing contact lines.
\newblock {\em Physics of Fluids}, 17(8):082106l, 2005.

\bibitem{Howell}
P.~D. Howell.
\newblock Surface-tension-driven flow on a moving curved surface.
\newblock {\em Journal of Engineering Mathematics}, 45:283--308, 2003.

\bibitem{JChiniK}
O.~E. Jensen, G.~P. Chini, and J.~R. King.
\newblock Thin-film flows near isolated humps and interior corners.
\newblock {\em Journal of Engineering Mathematics}, 50:289--309, 2004.

\bibitem{KalliadasisHomsy}
S.~Kalliadasis, C.~Bielarz, and G.~M. Homsy.
\newblock Steady free-surface thin film flows over topography.
\newblock {\em Physics of Fluids}, 12(8):1889--1898, 2000.

\bibitem{MyersRev}
T.~G. Myers.
\newblock Thin films with high surface tension.
\newblock {\em SIAM Review}, 40:441--462 (electronic), 1998.

\bibitem{MyersSolidII}
T.~G. Myers, J.~P.~F. Charpin, and S.~J. Chapman.
\newblock The flow and solidification of a thin fluid film on an arbitrary
  three-dimensional surface.
\newblock {\em Physics of Fluids}, 14:2788--2803, 2002.

\bibitem{MyersSolidI}
T.~G. Myers, J.~P.~F. Charpin, and C.~P. Thompson.
\newblock Slowly accreting ice due to supercooled water impacting on a cold
  surface.
\newblock {\em Physics of Fluids}, 14:240--256, 2002.

\bibitem{Richards}
L.~A. Richards.
\newblock Capillary conduction of liquids in porous medium.
\newblock {\em Physics}, 1:318--333, 1931.

\bibitem{Royetal}
R.~V. Roy, A.~J. Roberts, and M.~E. Simpson.
\newblock A lubrication model of coating flows over a curved substrate in
  space.
\newblock {\em Journal of Fluid Mechanics}, 454:235--261, 2002.

\bibitem{StockerHosoi}
R.~Stocker and A.~E. Hosoi.
\newblock Lubrication in a corner.
\newblock {\em Journal of Fluid Mechanics}, 544:353--377, 2005.

\bibitem{TuckRev}
E.~O. Tuck and L.~W. Schwartz.
\newblock A numerical and asymptotic study of some third-order ordinary
  differential equations relevant to draining and coating flows.
\newblock {\em SIAM Review. A Publication of the Society for Industrial and
  Applied Mathematics}, 32(3):453--469, 1990.

\bibitem{wilson}
S.~D.~R. Wilson.
\newblock The drag-out problem in film coating theory.
\newblock {\em Journal of Engineering Mathematics}, 16(3):209--221, 1982.

\bibitem{WilsonJones}
S.~D.~R. Wilson and A.~F. Jones.
\newblock The entry of a falling film into a pool and the air-entraintment
  problem.
\newblock {\em Journal of Fluid Mechanics}, 128:219--230, 1983.

\end{thebibliography}
